\documentclass{amsart}

\usepackage{amssymb, amsthm, amsmath, latexsym, amsfonts, xypic}

\usepackage{epstopdf}
\usepackage{graphicx}
\theoremstyle{plain}

\newtheorem{theorem}{Theorem}[section]

\newtheorem{proposition}[theorem]{Proposition}

\theoremstyle{definition}

\newtheorem{definition}[theorem]{Definition}

\newtheorem{remark}[theorem]{Remark}
\numberwithin{equation}{section}

\newcommand{\Hom}{\operatorname{Hom}}

\newcommand{\Q}{{\mathbb{Q}}}

\newcommand{\Z}{{\mathbb{Z}}}

\newcommand{\kk}{{\mathbf{k}}}

\newcommand{\id}{\operatorname{id}}         % identity

\newcommand{\im}{\operatorname{im}}        % image

\newcounter{llistadepth}

\newenvironment{manlist}[1]{\addtocounter{llistadepth}{1}
      \edef\llistacontador{llista\romannumeral\the\value{llistadepth}}
      \list{({#1{\llistacontador}})}{\usecounter{\llistacontador}
      \def\makelabel##1{\hss\llap{##1}}
      \itemsep=2pt\parsep=0pt\topsep=3pt plus 1pt minus 1 pt}}{\endlist
      \addtocounter{llistadepth}{-1}}

\newenvironment{numlist}{\begin{manlist}{\arabic}}{\end{manlist}}

\usepackage{xypic}

\author{Hossein Abbaspour}

\begin{document}

\address{Laboratoire Jean Leray, Universit\'e de Nantes,\\ 2, rue de la Houssini\`ere  \\ Nantes 44300, France\\
email:\,\tt{abbaspour@univ-nantes.fr}} %\email{abbaspour@univ-nantes.fr}

\title{On the Hochschild homology of  open Frobenius algebras}

\maketitle

\begin{abstract} We prove that the shifted Hochschild chain complex $C_*(A,A)[m]$ of a symmetric open Frobenius algebra $A$ of degree $m$ has a natural homotopy coBV-algebra structure. As a consequence $HH_*(A,A)[m]$ and $HH^*(A,A^\vee)[-m]$ are respectively coBV and BV algebras. The underlying coalgebra and algebra structure may not be resp. counital and unital. We also introduce a natural homotopy BV-algebra structure on $C_*(A,A)[m]$ hence a BV-structure on $HH_*(A,A)[m]$. Moreover we prove that the product and coproduct on $HH_*(A,A)[m]$ satisfy the Frobenius compatibility condition i.e. $HH_*(A,A)[m]$ is an open Frobenius algebras.  If $A$ is commutative, we also introduce a natural BV structure on the shifted relative Hochschild homology $\widetilde{HH}_*(A)[m-1]$. We conjecture that the product of this BV structure is identical to the Goresky-Hingston\cite{GH} product on the cohomology of free loop spaces when $A$ is a commutative cochain algebra model  for $M$.
\end{abstract}
\tableofcontents
\section{Introduction}

There have been many important works on providing algebraic models for the string topology operations introduced by Chas-Sullivan (\cite{CS1,CS2}) and Cohen-Godin (\cite{CG}).  One approach is to use the Hochschild cohomology of closed Frobenius algebras \cite{CJ, Mer,Tradler, TZ,Kauf1, Kauf2, Me, WahWest}. In particular  F\'elix-Thomas \cite{FT2}  proved that  over rationals and for any closed simply connected manifold $M$ the Chas-Sullivan BV-algebra $H_*(LM)$ is isomorphic to  $HH^*(A):=HH^*(A,A^\vee)$ where $A$ is a finite dimensional model (i.e. closed Frobenius algebra) for the cochains algebra $M$. 

 In \cite{KS} Kontsevich-Soibelman constructed an action of the chains of moduli spaces of Riemann surfaces on the Hochschild complex of a closed Frobenius algebras. This is a special of Constello's
 theorem \cite{CostelloCY} for Calabi-Yau categories and induces  a natural BV and coBV structure on the the Hochschild homology  and the Frobenius compatibility  between the BV product and coBV coproduct.
 
 In this paper  we assume that  $A$ is a symmetric open Frobenius algebra (unital) therefore $A$  is not  necessarily endowed with a non-degenerate scaler product. Instead $A$ is equipped with a compatible pair of product  and coproduct of degree $m$. First we prove that the shifted Hochschild chain $C_*(A,A)[m]$  is naturally a homotopy coBV-algebra (Section 3) therefore $HH_*(A,A)[m]$ is a BV-algebra. Also as a consequence  $C^*(A,A)[-m]$  and $HH^*(A,A^\vee)$ are respectively homotopy BV-algebra and BV-algebra thus we recover Tradler's \cite{Tradler} result for the closed Frobenius algebras. In Section 4 we prove that $C_*(A,A)[m]$  has a natural homotopy BV-structure and $H_*(A,A)[m]$  is a BV-algebra. Moreover, in Section 5 we prove that the product and coproduct on $HH_*(A,A)[m]$ satisfy the Frobenius compatibility as well.  Such a compatibility was expected in the light of  Cohen-Godin work for the free loop spaces (of not necessarily closed manifolds).
 
 It is worth mentioning that $HH_*(A,A)$ is generally not a unital algebra (or equivalently $HH^*(A,A^\vee)$ is not counital), reflecting the fact 
 (in the geometric side) the free loop spaces are infinite dimensional manifolds thus their homology are not conunital.   We recall that a unit for  the Chas-Sullivan algebra on $H*(LM)$ exist if and only if the underlying manifold is closed manifold in which case  the cycle consisting of constant loops is the unit. Similarly  $HH_*(A,A)$ is not counital (or equivalently $HH^*(A,A^\vee)$ is not unital) because the underlying manifold of the cochain complex $A$ may not be  a closed compact one.

We also identify a natural BV-product on the shifted  relative Hochschild homology $HH_*(A,A)[m-1]$ . We believe  that this product is an algebraic model of the Goresky-Hingston \cite{GorHing} product on the relative cohomology $H^*(LM, M)$. As far as we know this product was not known even for closed Frobenius algebras.

Our results can potentially be used to give an algebraic model for the string topology of not necessarily closed manifolds.  That would require generalizing our results to an appropriate homotopic setting.  

Throughout this paper $\kk$ is a commutative ring and  $A=\kk\oplus \bar{A}$ is a positively graded augmented unital differential associative $\kk$-algebra with $\deg d_A=+1$, $\bar{A}=A/\kk$ i.e.  $\bar{A}$ is the kernel of the augmentation $\epsilon: A \rightarrow \kk$.  See the appendix for the sing conventions.

\textbf{Acknowledgment:} I am grateful to Nathalie Wahl for many helpful communications. 

\section{Open Frobenius algebras and BV-Algebras}
We use the sing conventions explained in the appendix. A differential graded  $(A,d)$-module, or $A$-module for short, is a $\kk$-complex $(M,d)$ together with an (left) $A$-module structure $\cdot: A\times M\rightarrow M$ such that $ d_M(ax)=d_A(a)x+(-1)^{|a|}a d_M(x)$. The multiplication map is of degree zro \emph{i.e.} $\deg(ax)=\deg a+\deg x $. In particular, the identity above implies that the differential of $M$ has to be of degree $1$.

Similarly for a $(M,d_M)$ a differential graded $(A,d)-bimodule$, we have
$$
d_M(axb)=d_A(a)xb+(-1)^{|a|}ad_M (x)b+(-1)^{|a|+|m|}axd_Ab,
$$
or equivalently, $M$ is a $(A^e:=A\otimes A^{op}, d_A\otimes 1 +1\otimes d_A)$ DG-module where $A^{op}$ is the algebra whose
underlying graded vector space is $A$ with the opposite multiplication of $A$, \emph{i.e.} $a\overset{op}{\cdot}b=(-1)^{|a|.|b|}b\cdot a$. All modules considered in this article are differential modules. We will also drop the indices from the differential when there is no possibility of confusion.

\begin{definition}(DG open Frobenius algebra)\label{defopfrob}.  A \textit{differential graded open Frobenius $\kk$-algebra} of degree $m$ is a triple $(A,\cdot,\delta)$  such that:
\begin{enumerate}
\item $(A,\cdot)$ is a unital differential graded associative algebra whose product has degree zero,
\item $(A, \delta)$ is a differential graded coassociative coalgebra  of degree $m$ that is $(\delta\otimes 1)\delta=(-1)^m(1\otimes \delta)\delta$ and $\delta$ is chain map of degree $m$, \emph{i.e.} $\delta d=(-1)^m(d\otimes 1+1\otimes d)\delta$ .

\item $\delta : A\rightarrow A\otimes A$ is a right and left differential $A$-module map. Using (simplified) Sweedler's notation, this reads
$$
\sum_{(xy)} (xy)'\otimes (xy)''=\sum_{(y)}(-1)^{m|x|} xy'\otimes y''= \sum_{(x)}  x'\otimes x''y.
$$
\end{enumerate}
Here we have simplified Sweedler's notation for the coproduct $ \delta x=\sum_{i} x'_i\otimes x''_i$, to   $ \delta x=\sum_{(x)} x'\otimes x''$ where $(x)$ should be thought of as the index set for $i$'s.  Since the coproduct is  assumed to have degree $m$ therefore $\deg x' +\deg x''=m$.

In particular we have
$$
\sum_{(dx)} (dx)'\otimes(dx)'' =(-1)^m ( \sum_{(x)} (dx'\otimes x''+(-1)^{|x'|} x'\otimes dx''))
$$
and
$$
\sum (x')'\otimes (x')''\otimes x''=\sum (-)^{m|x'|+m}x'\otimes (x'')'\otimes (x'')''
$$
We shall say $(A,\cdot,\delta)$ is symmetric if $ \sum_{(1)} 1'\otimes 1''= \sum_{(1)}(-1)^{|1'||1''|+m} 1''\otimes 1'$. 
\end{definition}

We recall that a closed (DG) Frobenius algebra is a finite dimensional unital associative differential graded $\kk$-algebra $A=\oplus_{i\geq 0} A_i$ equipped with a \emph{symmetric inner product}  $\langle -,- \rangle$ 
such that the map $ \alpha: x\mapsto (y\mapsto  \alpha_x(y):=(-1)^{|x|+m} \langle x, y \rangle=(-1)^{|y|} \langle x, y \rangle)  $, from $A$ to $A^\vee$ is a degree $m$ isomorphism of differential graded $A$-bimodules.  

We recall that symmetric means $$ \langle x,y \rangle=(-1)^{|x||y|}\langle y,x \rangle=(-1)^{|x|(m-|x|)}\langle x,y\rangle.$$

Notice that $\alpha$ is of degree $m$ therefore $\alpha$ being $A$-biequivarant must take into account the degree.  We spell this out in details since it is important to get the signs right. 

Let $L:A\otimes A^\vee\to A^\vee$ and $R:A^\vee\otimes A\to A^\vee$ be respectively the left and right  action of $A$ on  $A^\vee$
We will use the same notation for the action of $A$ on itself. Then $\alpha$ being $A$-biequivariant means
$$
L(1\otimes \alpha)=\alpha \circ L
$$
and 
$$
R(\alpha \otimes 1)=\alpha \circ R,
$$
which after applying to elementa $x,y\in A$, it reads $(-1)^{m|y|}y\cdot \alpha_x=\alpha_{yx}$ and $  \alpha_x\cdot y  =\alpha_{xy}$.  
It follows from the definition of closed Frobenius algebras that the inner product is invariant
$$
\langle xy,z \rangle=\langle x,yz\rangle,
$$ 
and symmetric and 
 \begin{equation}\label{eqclosed}
  \langle dx,y \rangle=-(-1)^{|x|}\langle x,dy \rangle.
 \end{equation}
  
We can now define a coproduct $\delta: A\to A\otimes A$ by requiring that the diagram

\begin{equation}\label{diagfrob}
\xymatrix{ A\ar[r]^{\alpha}\ar[dd]^-{\delta} & A^\vee \ar[rr]^-{\text{dual of the product}} &&(A \otimes A)^\vee \\
&  A^\vee\otimes A^\vee \ar[urr]_-{i_{A,A}}&  &\\
A\otimes A \ar[ur]_-{\alpha\otimes \alpha}}
\end{equation} 
to be commutative. Note that in the diagram above $\alpha$, $\alpha\otimes\alpha$ and  $i_A\otimes i_A$ (because $dimA <\infty$) are isomorphisms therefore $\delta$  exists and is unique because of  the non-degeneracy of the inner product. The coproduct $\delta x=\sum_{(x)}x'\otimes x''$  is characterized by the identity

 \begin{equation}\label{eqclosedco}
\langle x,ab \rangle =  \sum_{(x)} (-1)^{m |x'|} \langle x'',a \rangle \langle x',b \rangle \end{equation}
Since the inner product has  degree $m$, we  obtain
 \begin{equation}
\langle x,ab \rangle =(-1)^{m|b|+m}  \sum_{(x)}  \langle x'',a \rangle \langle x',b \rangle, \end{equation}
which in the special case  $x=1$ it reads,

 \begin{equation}\label{eqclosedco3}
\langle a,b \rangle=  \langle 1,ab \rangle =(-1)^{m|b|+m}  \sum_{(1)}  \langle 1'',a \rangle \langle 1',b \rangle \end{equation}

The coproduct $\delta$  is coassociative of degree $m$ and satisfies condition (3) of Definition \ref{defopfrob} because all the other maps in the diagram (\ref{diagfrob}) are morphisms of $A$-bimodules.
We can also check this directly,
\begin{equation}\begin{split}
\sum(-1)^{m|(xy)'|} \langle  (xy)'',a \rangle \langle (xy)',b \rangle= \langle xy,ab \rangle = \langle x,yab \rangle \\
= \sum (-1)^{m |x'|} \langle x'',ya \rangle \langle x',b \rangle\\
=\sum (-1)^{m|x'|} \langle x''y,a \rangle \langle x',b \rangle
\end{split}
\end{equation}
which together with the non-degeneracy of the inner product imply
$$
\sum_{(xy)} (xy)'\otimes (xy)''=\sum_{(y)}x'\otimes x'' y,
$$
Similarly 
\begin{equation}\begin{split}
\sum(-1)^{m|(xy)'|} \langle  (xy)'',a \rangle \langle (xy)',b \rangle= \langle xy,ab \rangle = (-1)^{|x|(m-|x|)}\langle y,abx \rangle \\
= (-1)^{|x|(m-|x|)}\sum (-1)^{m |y'|} \langle y'',a \rangle \langle y',b x\rangle\\
=\sum (-1)^{m|y'|} \langle y'',a \rangle \langle xy',b \rangle=\sum (-1)^{m|x|}(-1)^{m|xy'|} \langle y'',a \rangle \langle xy',b \rangle
\end{split}
\end{equation}
therefore
$$\sum_{(xy)} (xy)'\otimes (xy)''=\sum_{(y)}(-1)^{m|x|}xy'\otimes y'' $$
In other words a closed Frobenius algebra over a field is also an open Frobenius algebra. 

By replacing  $b=1$ in (\ref{eqclosedco}), we obtain

 \begin{equation*}
 \begin{split}
\langle x,a \rangle & =\langle x,a1  \rangle=  \sum_{(x)}(-1)^{m|x'|}  \langle x'',a \rangle \langle x',1 \rangle\\&
=\sum_{(x)} (-1)^{m|x'|}  \langle  \langle x',1 \rangle  x'',a \rangle.
\end{split}
\end{equation*}
 which implies 
  \begin{equation}
 x= \sum_{(x)} (-1)^{m|x'|} \langle x',1 \rangle  x''.
 \end{equation}
Similarly by taking  for $a=1$, we obtain

 \begin{equation*}
 \begin{split}
\langle x,b \rangle & =\langle \sum (-1)^{m|x'|} \langle x'',1 \rangle  x',b \rangle.
\end{split}
\end{equation*}
The non-degeneracy of the inner product implies that
 \begin{equation}
 x=\sum_{(x)}  (-1)^{m|x'|}  \langle x'',1 \rangle  x'
  \end{equation}
In other words $\eta(x)=\langle x,1\rangle$ is a counit  that is 
\begin{equation}\label{eq-cou}
x=\sum_{(x)}  (-1)^{m|x'|}\eta(x')x''=\sum_{(x)}  (-1)^{m|x'|}\eta(x'')x'
\end{equation}
\bigskip
Again using the Frobenius property  (or a diret computation) we have 

\begin{equation}\label{eq-cour}
x=\sum_{(1)}  (-1)^{m|1'|}\eta(x1')1''=\sum_{(1)}  (-1)^{m|1'|}\eta(1''x)1'
\end{equation}
and 
\begin{equation}\label{eq-cour1}
x=\sum_{(1)}  (-1)^{|1'|}\eta(1'x)1''=\sum_{(1)}  (-1)^{|1'|}\eta(x1'')1'
\end{equation}

\noindent \textbf{Example 1:} An important example of symmetric open Frobenius algebra is the cohomology with compact support  $H_c^*(M,\Z)$ of an oriented $n$-dimensional manifold $M$  (not necessarily closed).  
Note that $H_c^*(M,\Z)$  is equipped with the usual cup product and, using the Poincar\'e duality isomorphism (see~\cite[Theorem 3.35]{Hatcherbook})
$$
H^*_c(M,\Z) \simeq H_{n-*}(M,\Z),
$$
one can transfer the natural coproduct $\Delta$ from $H_*(M,\Z)$ to  $H_c^*(M,\Z)$.   Then the triple $(H^*_c(M,\Z) , \cup, \Delta)$ is a symmetric open Frobenius algebra whose differential is identically zero. The Frobenius compatibility condition is satisfied because the Poincar\'e duality isomorphism above is a map of $H_c^*(M,\Z)$-modules.  Note that this open Frobenius structure is only natural with respect to proper maps or inclusion (of open sets).
 
If $M$ is closed then $H^*(M,\Z)=H^*_c(M,\Z)$ is indeed a closed Frobenius algebra since $H^*(M,\Z)$ has a counit given by $\int :H^*(M)\to \Z$, the evaluation on the fundamental class of $M$, while Poincar\'e duality is given by capping with the fundamental class.  The non-degenerate inner product is defined by $\langle x,y\rangle:= \int_{[M]} x\cup y$. Over  the rationals it is possible to lift this Frobenius algebra structure to the level of cochains. By a result of Lambrechts and Stanley~\cite{LamStan}, over rationals there is a connected finite dimensional commutative DG algebra $A$ which is quasi-isomorphic to the singular cochain algebra $C^*(M,\Q)$ on a given $n$-dimensional manifold $M$,  equipped with a bimodule isomorphism $A\rightarrow A^\vee $ inducing the Poincar\'e duality $H^*(M)\rightarrow H_{n-*}(M)$. The analogue  of this result for open manifolds  is still not known. 
It is more reasonable to expect a kind of homotopy open Frobenius model for open manifolds rather than an open Frobenius algebra model. This also suggests 
the result of this paper need to be generalized to the homotopy Frobenius algebra, a notion to be defined.
\bigskip

 \noindent \textbf{Example 2:} 
 It is known that the homology of the free loop space of a closed oriented manifolds is an open Frobenius algebra \cite{CG}. Similarly the Hochschild cohomology $HH^*(A)$ of a closed Frobenius  algebra is  an open Frobenius algebra \cite{TZ}.

\begin{proposition}\label{pro-frob}
\begin{numlist} 
\item A closed Frobenius algebra is a symmetric open Frobenius algebra.
\item A symmetric open Frobenius  algebra $A$ with a counit is finite dimensional and in fact is a closed Frobenius algebra.
\item A symmetric and commutative open Frobenius algebra is cocommutative.
\item For all $z$ in a symmetric open Frobenius  algebra $A$,  $\sum_{(z)}(-1)^{|z''||z'|} z''z'$ belongs to the center of  $A$.
\end{numlist}
\end{proposition}

\begin{proof}
1) It follows from the characterization (\ref{eqclosedco3}). Indeed,

 \begin{equation*}\begin{split}
&(-1)^{m|b|+m}   \sum_{(1)} (-1)^{|1'||1''|+m}  \langle 1',a \rangle \langle 1'',b \rangle=(-1)^{m|b|}    \sum_{(1)} (-1)^{(m-|a|)(m-|b|)}  \langle 1',a \rangle \langle 1'',b \rangle\\
&=  (-1)^{|a||b|+m|a|+m}    \sum_{(1)}  \langle 1',a \rangle \langle 1'',b \rangle 
\\&=(-1)^{|a||b|} \langle b,a\rangle= \langle a,b\rangle
\end{split}
 \end{equation*}
therefore $\sum_{(1)}  (-1)^{|1'||1''|+m} 1''\otimes 1'=\sum_{(1)}   1'\otimes 1''$.

2)  The inner product is defined by $\langle x,y\rangle=\eta(xy)$. It is clearly invariant. The identity $x= \sum (-1)^{m|x'|}\eta (x'')x'=\sum (-1)^{m|1'|}\eta (x1'')1'$ proves that $A=Span_{\kk} \{1's\}$ hence the finite dimensionality of $A$.

Now we must prove that $\langle-,-\rangle$ is symmetric.  By the identity (\ref{eq-cou}) we can write $$xy=\sum (-1)^{m|x'|}\eta(x'') x' y=\sum (-1)^{m|1'|}\eta(1''x) 1' y,$$ therefore for all $x$ and $y$
\begin{equation*}
\langle x,y\rangle=\sum (-1)^{m|1'|}\eta(1''x) \eta(1' y).
\end{equation*}
Since $A$ is symmetric, we have
\begin{equation*}\begin{split}
\langle x,y\rangle&=\sum (-1)^{m|1'|}\eta(1''x) \eta(1' y)=\sum (-1)^{m|1''|+|1'||1''|+m}\eta(1'x) \eta(1'' y)\\
&=\sum (-1)^{m (m-|1'|)+ (m-|x|)(m-|y|))+m}\eta(1'x) \eta(1'' y)\\
&=\sum (-1)^{m|1'|+|x||y|}\eta(1'x) \eta(1'' y)=(-1)^{|x||y|}\langle y,x\rangle
\end{split}
\end{equation*}
3) 
\begin{equation}\begin{split}
&\sum x'\otimes x''=\sum (-1)^{m|x|}x 1'\otimes 1 ''=\sum (-1)^{m|x|+|1''||1'|+m} x1''\otimes 1'\\&=  \sum  (-1)^{m|x|+|1''||1'|-|x||1''|+m}  1''x\otimes 1'=\sum (-1)^{m|x|+(|x''|-|x|)|x'|-|x|(|x''|-|x|)+m} x''\otimes x'\\
&=\sum (-1)^{|x'||x''|+|x|(m+|x|-|x'|-|x''|)+m} x''\otimes x'=\sum (-1)^{|x'||x''|+m} x''\otimes x'.
 \end{split}
\end{equation}
\newline
4) 
 \begin{equation*} \begin{split}
x\sum_{(z)} (-1)^{|z''||z'|} z''z'&=\sum_{(z)}(-1)^{|z''||z'|}  xz''z'= \sum_{(1)}(-1)^{(|1''|+|z|)|1'|}x 1''z1'\\
 &=\sum_{(1)}(-1)^{(|1'|+|z|)|1''| +|1'||1''|+m}x 1'z1''=\sum_{(1)}(-1)^{|z||1''|+m}x 1'z1''\\
 &=\sum_{(x)}(-1)^{m|x|+|z| |x''|+m} x'zx'' =\sum_{(x)}(-1)^{m|x|+|z| (|1''|-|x|)+m}   1'z1'' x
\\ &=\sum_{(1)} (-1)^{m|x|+|z| (|1'|-|x|)+|1'||1''|}  1''z1' x\\ 
& = \sum_{(1)} (-1)^{m|x|+|z| (|z'|-|x|)+|z'|(|z''|-|z|)}  z''z' x\\
&=(-1)^{(m+|z|)|x|}(\sum_{(z)}(-1)^{|z''||z'|} z''z')x.
 \end{split}
\end{equation*}
We recall that  $\sum_{(z)}(-1)^{|z''||z'|} z''z'$ is of degree $m+|z|$.
\end{proof}

Before explaining how an open Frobenius algebras gives rise to a (co)BV-algebra, we recall the definition of the BV-algebras and the definition of Hochschild homology and cohomology.

\begin{definition}\label{BV-def}(Batalin-Vilkovisky algebra) A BV-algebra is a Gerstenhaber algebra  $(V^*,\cdot,[-,-])$ with a degree  one operator
$\Delta: V^*\rightarrow V^{*+1}$ whose deviation from being a derivation for the product $\cdot$ is the Gerstenhaber bracket $[-,-]$, \textit{i.e.}
$$ [a,b]:=(-1)^{|a|}\Delta(ab)-(-1)^{|a|}\Delta(a)b-a \Delta(b), $$
and $\Delta^2=0$.
\end{definition}

It follows from $\Delta^2=0$ that $\Delta$ is a derivation for the bracket.
In fact the Leibniz identity for $[-,-]$ is equivalent to the 7-term relation \cite{Getz}

\begin{equation}\label{7term}
\begin{split}
\Delta(abc)&= \Delta (ab)c +(-1)^{|a|} a\Delta(bc)+ (-1)^{(|a|-1)|b|}b\Delta (ac)\\
  &-\Delta (a) bc - (-1)^{|a|}a \Delta(b) c-(-1)^{|a|+|b|}ab\Delta c.
\end{split}
\end{equation}

Definition \ref{BV-def} is equivalent to the following one:

\begin{definition}
A BV-algebra is a graded commutative associative algebra  $(A^*,\cdot)$ equipped with a degree one operator $\Delta: A^*\rightarrow A^{*+1}$ which satisfies the 7-term relation (\ref{7term}) and $\Delta^2=0$. It follows from the 7-term relation that $[a,b]:=(-1)^{|a|}\Delta(ab)-(-1)^{|a|}\Delta(a)b-a \Delta(b)$ is a Gerstenhaber bracket for the graded commutative associative algebra $(A^*,\cdot)$.
\end{definition}

 There are very interesting examples using the differential forms of  Riemannian  or symplectic manifolds, which are essentially due to Kozsul \cite{KoszulBV}. The inspiring example  for us is the homology of the free loop space $LM:=C^0(S^1,M)$ of an oriented manifold \cite{CS1} for which an algebraic model can be obtained using Hochschild cohomology of cochains algebras of  $M$ \cite{Jones}. Let us recall the definition of the Hochschild complex.

 The \emph{(normalized) Hochschild chain complex} with coefficients in $M$ is defined to be 

\begin{equation}
C_*(A,M):= M \otimes  T(s \bar{A})
\end{equation}
and comes equipped with a degree +1 differential $ D_{Hoch}=d_{0}+d_{1}$.  We recall that $TV=\oplus_{n\geq 0} V^{\otimes n}$ denotes the tensor algebra of a $\kk$-module $V$.

The internal differential is given by
\begin{equation}
\begin{split}
d_{0}(m  [a_{1},\cdots , a_{n}])&= d_M m  [a_{1},\cdots, a_{n}]-\sum_{i=1}^{n} (-1)^{\epsilon_{i}}m  [a_{1},\cdots,d_A a_{i},\dots a_{n}]\\
\end{split}
\end{equation}
and the external differential is
\begin{equation}
\begin{split}
d_{1}(m [a_{1}, \cdots , a_{n}] )= & (-1)^{|m|} m a_{1} [a_{2}, \cdots , a_{n}]\\ 
&\qquad + \sum_{i=2}^{n} (-1)^{\epsilon_{i}} m [a_{1},
\cdots, a_{i-1}a_{i}, \cdots ,a_{n}]\\ 
&\qquad \qquad -(-1)^{\epsilon_{n}(|a
_n|+1)} a_{n}m [a_{1}, \cdots ,  a_{n-1}],
\end{split}
\end{equation}
with $\epsilon_0=|m|$ and $\epsilon_{i}=|m|+|a_{1}|  +\cdots |a_{i-1}|-i+1 $ for $i\geq 1$. Note that the degree of  $m [a_{1}, \cdots , a_{n}]$ is $|m|+\sum_{i=1}^{n}|a_i|-n$.

When $M=A$, by definition $(C_*(A,A),D_{Hoch}=d_0+d_1)$ is the \emph{Hochschild chain complex of} $A$ and  the Hochschild homology of $A$ 
is by definition $HH_*(A,A):=\ker D/\im D$ is the Hochschild cohomology of $A$. 

Similarly we define the $M$-valued \emph{Hochschild cochain} of $A$ to be $$ C^*(A,M):=\Hom_{\kk}(T(s\bar{A}),M). $$ For a homogenous cochain complex$f\in C^n(A,M)$, the degree $|f|$ is defined to be the degree of the linear
map $f: (s\bar{A})^{\otimes n}\rightarrow M$. In the case of Hochschild cochains, the external differential of $f\in \Hom (s\bar{A}^{\otimes n},
M)$ is
\begin{equation}
\begin{split}
d_{1}(f)(a_{1}, \cdots , a_{n})&=-(-1)^{(|a_1|+1)|f|}a_{1}f(a_{2}, \cdots , a_{n})+\\ &-\sum_{i=2}^{n}  (-1)^{\epsilon_i} f(a_{1}, \cdots ,
a_{i-1}a_{i}, \cdots, a_{n})+(-1)^{\epsilon_n} f( a_{1}, \cdots , a_{n-1})a_n,
\end{split}
\end{equation}
where $\epsilon_i=|f|+|a_1|+\cdots +|a_{i-1}|-i+1$. The internal differential of $f\in C^*(A,M)$ is
\begin{equation}
\begin{split}
d_{0}f(a_{1}, \cdots , a_{n})&=d_Mf(a_{1}, \cdots , a_{n})+\sum_{i=1}^n (-1)^{\epsilon_i} f(a_{1}, \cdots, d_A a_i ,\cdots , a_{n}).
\end{split}
\end{equation}
The Hochschild cohomology of $A$  with coefficient in $M$ is  by definition $HH^*(A,M):=\ker D^{Hoch}/\im D^{Hoch}$.
\begin{remark}
Naturally  one can consider $\kk$-dual $\Hom_{\kk}(C_*(A,A),\kk), D_{Hoch}^\vee)$ and  its cohomology $\ker  D_{Hoch}^\vee/  \im D_{Hoch}^\vee $.
 The result is isomorphic to $(C^*(A,A^\vee),D^{Hoch}) $.  In fact the isomorphism $\widetilde{}:  (C^*(A,A^\vee),D^{Hoch}) \to  \Hom_{\kk}(C_*(A,A),\kk), D_{Hoch}^\vee) $ is given by $f\mapsto \tilde{f}$,
\begin{equation}\label{identif-eq}
\widetilde{f}(a_0,a_1,\cdots, a_n) =(-1)^{(|a_0| +1)|f|}f(a_1,a_2,\cdots, a_n)(a_0)
\end{equation}
Therefore  $HH^*(A,A^\vee)\simeq H^*(Hom_{\kk}(C_*(A),\kk), D_{Hoch}^\vee))$.
All over this article $C^*(A,A^\vee)$ is identified with  $\Hom (C_*(A,A),\kk)$ using the isomorphism above.
\end{remark}

\subsubsection*{Gerstenhaber bracket and cup product:}
When $M=A$, for $x\in C^m(A,A)$ and $y\in C^n(A,A)$ one defines the \emph{cup product} $x\cup y\in C^{m+n}(A,A)$ and the \emph{Gerstenhaber
bracket}  $[x,y] \in C^{m+n-1}(A,A)$ by
\begin{equation}
(x\cup y)(a_{1},\cdots , a_{m+n}):=(-1)^{|y|(\sum_{i\leq m} |a_{i}|+1)}x(a_{1},\cdots , a_{m})y(a_{n+1}, \cdots , a_{m+n}),
\end{equation}
 and
 \begin{equation}\label{def-circ}
 [x,y]:= x\circ y - (-1)^{(|x|+1)(|y|+1)}y\circ x,
 \end{equation}
 where
 $$
 (x\circ_j y)( a_{1},\cdots ,a_{m+n-1})=(-1)^{(|y|+1)\sum_{i\leq j} (|a_{i}|+1)} x(a_{1}, \cdots , a_{j},  y(a_{j+1}, \cdots , a_{j+m}),\cdots  ).
 $$
 and 
 \begin{equation}
 x\circ y=\sum_j x\circ_j y
 \end{equation}
 It turns out that the operations $\cup$ and $[-,-]$ are chain maps, hence they define two well-defined operations on $HH^*(A,A)$. Moreover,  $\cup$ is commutative up to homotopy
which is given by $-\circ-$. 
\begin{theorem}\label{Gersten}(Gerstenhaber \cite{Gers}) Let $A$ be a differential graded associative algebra. $(HH^*(A,A),\cup,[-,-])$  is a \textit{Gerstenhaber algebra} that is for all $x,y$ and $z\in HH^*(A,A)$ we have:
\begin{enumerate}
    \item $\cup$ is an associative and graded commutative product,
    \item  $[x,y\cup z] = [x,y]\cup z + (-1)^{(|x|-1)|y|}y\cup[x,z] $ (Leibniz rule),
    \item $[x,y] = -(-1)^{(|x|-1)(|y|-1)} [y,x] $,
    \item $[[x,y],z] = [x,[y,z]] -(-1)^{(|x|-1)(|y|-1)}[y,[x,z]] $ (Jacobi identity).
\end{enumerate}
\end{theorem}
The Hochschild homology and cohomology of an algebra have an extra feature, which is the existence of the \emph{Connes operators} $B$, respectively $B^\vee$ (\cite{connes}). On the chains we have
\begin{equation}\label{B1}
    B(a_0  [a_1, a_2 \cdots, a_n ])=\sum_{i=1}^{n+1}(-1)^{\epsilon_i} 1[a_{i+1} \cdots a_n,a_0,\cdots,a_i]
\end{equation}
and on the dual theory $C^*(A)=\Hom_\kk (T(s\bar{A}),A^\vee)=\Hom (A \otimes  T(s\bar{A}), \kk)$ we have 
$$ 
(B^\vee \phi)(a_0 [a_1, a_2 \cdots, a_n ])=(-1)^{|\phi|}\sum_{i=1}^{n+1}(-1)^{\epsilon_i} \phi (1[a_{i} \cdots a_n,a_0,\cdots,a_{i-1}] ), 
$$ 
where $\phi \in C^{n+1}(A)=\Hom(A\otimes (s\bar{A})^{\otimes n+1}  , \kk)$ and $\epsilon_i=(|a_0|+\dots |a_{i-1}|-i)(|a_i|+\dots |a_{n}|-n+i-1)$. In other words $$B^\vee (\phi)= (-1)^{|\phi|} \phi \circ B.$$
Note that $\deg(B)=-1 \text{ and } \deg B^\vee=+1.$ The following theorem shows how a closed Frobenius algebra gives rise to a BV-algebra.
 \begin{theorem} (Tradler \cite{Tradler})\label{Tradthm} The Hochschild cohomology $HH^*(A,A)$ of a Frobenius algebra $A$ has natural a BV-structure whose underlying  Gerstenhaber structure is the standard 
 one \cite{Gers}. The BV-operator corresponds to the Connes operator $B^\vee$ using the natural isomorphism  $ HH^*(A,A) \simeq HH^*(A,A^\vee)[m] $. 
  \end{theorem}
The main idea here is that we try to identify the  homotopy (co)BV-structures directly on $C_*(A ,A)$  (and its dual) rather than $C^*(A,A)$.  
\section{coBV structure on Hochschild homology}\label{section2}
In this section we present a natural homotopy coBV-structure on the shifted Hochschild chain complex $C_*(A,A)[m]$ of a symmetric open Frobenius algebra $(A,\cdot, \delta)$ of degree $m$. The natural candidate for the coBV operator is the Connes operator $B$, so we just need a  degree $m$ coproduct on the Hochshild chains $C_*(A,A)$. This is given by formula (before the shift)
\begin{equation} \label{eq-cup-frob}
\theta (a_0[a_1,\cdots,a_n])=\sum_{(a_0), 0\leq i\leq n}  +(-1)^{|a_0'|\sigma_i}(a_0'' [a_1,\cdots,a_{i-1},a_i ]) \otimes (a_0'  [a_{i+1}, \cdots,a_{n}])
\end{equation}
where $\sigma_i=a_0''+a_1+a_2+\cdots a_i+i$.  This coproduct is of degree $m$ and is a chain map \emph{i.e.}
\begin{equation*}
\theta D_{Hoch}=(-1)^m(D_{Hoch}\otimes1+1\otimes D_{Hoch})
\end{equation*}
The proof that $\theta$ is a chain map, uses $A$ being symmetric. The most nontrivial part of the proof  is that 
\begin{equation*}
\begin{split}
&\sum_{(a_0)}(-1)^{|a'_0|(|a''_0|+|a_1|+\cdots |a_p|+p)+(|a_p|+1)(|a_0''|+|a_1|+\cdots |a_{p-1}|+p-1)} a_pa''_0[a_1,\cdots, a_{p-1}]\otimes a'_0[a_{p+1},\cdots,a_n]\\
&=\sum_{(a_0)}(-1)^{|a'_0|(|a''_0|+|a_1|+\cdots |a_{p-1}|+p-1)+|a_0''|+|a_1|+\cdots +|a_{p-1}|+p-1+|a_0'|}a''_0[a_1,\cdots, a_{p-1}]\otimes a_0'a_p[a_{p+1},\cdots,a_n].
\end{split}
\end{equation*}
appears in $(D_{Hoch}\otimes1+1\otimes D_{Hoch})$ twice but with opposite signs. The proof of the identity above is as follows:
\begin{equation*}
\begin{split}
&\sum_{(a_0)}(-1)^{|a'_0|(|a''_0|+|a_1|+\cdots |a_p|+p)+(|a_p|+1)(|a_0''|+|a_1|+\cdots |a_{p-1}|+p-1)} a_pa''_0[a_1,\cdots, a_{p-1}]\otimes a'_0\\
&= \sum_{(a_0)}(-1)^{|1'|(|1''|+|a_0|+\cdots |a_p|+p)+(|a_p|+1)(|1''|+|a_0|+\cdots |a_{p-1}|+p-1)} a_p1''a_0[a_1,\cdots, a_{p-1}]\otimes 1'\\
&= \sum_{(a_0)}(-1)^{|1''|(|a_0|+\cdots |a_p|+p)+(|a_p|+1)(|1'|+|a_0|+\cdots |a_{p-1}|+p-1)+m} a_p1'a_0[a_1,\cdots, a_{p-1}]\otimes 1''\\
&= \sum_{(a_0)}(-1)^{m|a_p|+|a_p''|(|a_0|+\cdots |a_p|+p)+(|a_p|+1)(|a_p'|+|a_0|+\cdots |a_{p}|+p-1)+m} a'_pa_0[a_1,\cdots, a_{p-1}]\otimes a_p''\\
&= \sum_{(a_0)}(-1)^{m|a_p|+(|1''|+|a_p|)(|a_0|+\cdots |a_p|+p)+(|a_p|+1)(|1'|+|a_0|+\cdots |a_{p}|+p-1)+m} 1'a_0[a_1,\cdots, a_{p-1}]\otimes 1''a_p\\
&= \sum_{(a_0)}(-1)^{|1''||1'|+m|a_p|+(|1'|+|a_p|)(|a_0|+\cdots |a_p|+p)+(|a_p|+1)(|1''|+|a_0|+\cdots |a_{p}|+p-1)} 1''a_0[a_1,\cdots, a_{p-1}]\otimes 1'a_p\\
&= \sum_{(a_0)}(-1)^{(|a_0''|+|a_0|)|a_0'|+m|a_p|+(|a_0'|+|a_p|)(|a_0|+\cdots |a_p|+p)+(|a_p|+1)(|a_0''|+|a_1|+\cdots |a_{p}|+p-1)} a''_0[a_1,\cdots, a_{p-1}]\otimes a_0'a_p\\
&=\sum_{(a_0)}(-1)^{|a'_0|(|a''_0|+|a_1|+\cdots |a_{p-1}|+p-1)+|a_0''|+|a_1|+\cdots +|a_{p-1}|+p-1+|a_0'|}a''_0[a_1,\cdots, a_{p-1}]\otimes a_0'a_p. \\
\end{split}\end{equation*}
This gives rise to a product on Hochschild cochains as follows:  For  $\tilde{f},\tilde{g} \in  \Hom (A \otimes T(s\bar{A}), \kk)$  we set
$$
\tilde{f} \circ \tilde{g}= \mu(\theta^\vee (\tilde{f}\otimes \tilde{g}))=(-1)^{m(|\tilde{f} |+| \tilde{g}|)} \mu(\tilde{f}\otimes \tilde{g})\circ \theta$$
where $\mu: \kk\otimes \kk \rightarrow \kk$  is the multiplication.   Note that this product is of degree $-m$, therefore in order to obtain a product of degree zero we should shift the grading by $-m$.  The new degree zero product on  $\Hom(C_*(A,A),\kk)[-m]$ is  (see the appendix)
\begin{equation*}
\widetilde{f}\odot \widetilde{g}= (-1)^{m|\widetilde{f}|}\widetilde{f} \circ \tilde{g}
\end{equation*}
More explicitly, for $\tilde{f}$ and  $\tilde{g} \in \Hom(C_*(A,A),\kk)[-m]$ we have
$$(\tilde{f} \odot \tilde{g})(a_0 [a_1,\cdots,a_n])= \sum_{(a_0),1\leq  i\leq n}  (-1)^{m|\tilde{g}|+a'_0\sigma_i+(m+|\tilde{g}|)\sigma_i}\tilde{f}(a''_0 [a_1,\cdots,a_{i-1}, a_i]) \tilde{g}(a'_0  [a_{i+1}, \cdots,a_{n}]).$$ 
In the case of a closed Froebnius algebra this product corresponds to the standard cup product on $HH^*(A,A)$ using the isomorphism
\begin{equation*}
HH^*(A,A)\simeq HH^*(A,A^\vee)[-m]\simeq H^*(\Hom(C_*(A,A))[-m]
\end{equation*}
induced by the inner product on $A$.  More explicitly we identify  $A$ with $A^\vee$, as bimodules,
using the map $ x\mapsto  (\alpha_x:=(-1)^{x} \langle x,-\rangle )$ which identifies $C^*(A,A)$ with $C^*(A,A^\vee)$.  The latter itself is identified with  $\Hom(C_*(A,A),\kk)$  using the isomorphism (\ref{identif-eq}). Overall we have an isomorphism of cochain complexes which sends
$f \in C^*(A,A)$, $f:(sA)^{\otimes n}\to A$ to the cochain $\tilde{f}\in   \Hom (A\otimes (sA)^{\otimes n},\kk)[m]$,
\begin{equation*}
\tilde{f}(a_0,a_1,\cdots, a_n):=(-1)^{(a_0+1)f}\langle a_0, f(a_1,\cdots, a_n)\rangle.
\end{equation*}
Using the identity 
$$x= \sum (-1)^{m|x'|}\eta (x'')x'=\sum (-1)^{m|1'|}\eta (1''x)1'$$
we can write 
\begin{equation*}
f(a_1,\cdots, a_n):=\sum_{(1)}(-1)^{|\tilde{f}|(1+|1''|)} \tilde {f}(1'',a_1,\cdots, a_n)1' .
\end{equation*}
which is an explicit formula for the inverse of the isomorphism $f\mapsto \tilde{f}$.

Let  $f:  (s\bar{A})^{\otimes p} \to A$ and $g: (s\bar{A})^{\otimes q} \to A$  in $C^*(A,A)$ be two cochains. First note that  the degrees of $\tilde{f}$ and $\tilde{g}$ as elements of $\Hom (C_*(A,A),\kk)[m]$ are respectively equal to $|f|$ and $|g|$. 
\begin{equation*}
\begin{split}
&\widetilde{f \cup g}(a_0, a_1,\cdots, a_{p+q})= (-1)^{(a_0+1)(|f|+|g|)}\langle a_0,  ( f \cup g)( a_1,\cdots, a_{p+q})\rangle= \\
&\sum_{(a_0)}(-1)^{(a_0+1)(|f|+|g|)+m(1+|a'_0|)+|g|( \sigma_p-|a''_0|)}  \langle a''_0,  f ( a_1,\cdots, a_{p})\rangle\langle a'_0,  g( a_{p+1},\cdots, a_{p+q})\rangle =\\
&\sum_{(a_0)}(-1)^{(a_0+1)(|f|+|g|)+m(1+|a'_0|)+|g|( \sigma_p-|a''_0|)+(1+|a''_0|)|f|+(1+|a'_0|)|g|}  \tilde{f} (a''_0,  a_1,\cdots, a_{p})\tilde{g}(a'_0,  a_{p+1},\cdots, a_{p+q}). \\
\end{split}
\end{equation*}
Since  $\sigma_p=|f|+m$ it follows that $\widetilde{f \cup g}=\tilde{f}\ast \tilde{g}$, therefore $(HH^*(A,A),\cup)$ and $(H^*(A,A^\vee)[m], \odot)$ are isomorphic as algebras. 
\begin{theorem} \label{themop-cop} For a symmetric open Frobenius algebra $(A,\cdot, \delta)$ of degree $m$, the shifted Hochschild chain complex $(C_*(A,A)[m], \theta, D_{Hoch})$ is a homotopy co-BV algebra. As a consequence $(HH_*(A,A)[m], \theta, B)$ and $(HH^*(A,A^\vee)[m], \odot, B^\vee)$ are respectively BV-algebra and coBV-coalgebra. In particular if  $A$ is a closed Frobenius algebra, then using the natural isomorphism 
\begin{equation}\label{iso-clo-frob}
(C^*(A,A^\vee)[m], \odot) \simeq (C^*(A,A),\cup),
\end{equation}
$HH^*(A,A)$ is endowed with a BV-algebra whose underling Gerstenhaber algebra is the standard one and the BV-operator is the image of the Connes operator $B^\vee$ under the isomorphism (\ref{iso-clo-frob}). This recovers Tradler's result \cite{Tradler}, Theorem \ref{Tradthm} for closed Frobenius algebras.
\end{theorem}
\begin{proof}
The homotopy for co-commutativity is given by
\begin{equation*}\label{homotpy-cocom}
\begin{split}
h(a_0  [a_1, \cdots,a_n])&:=\sum_{(1), 0\leq i<j\leq n+1} (-1)^{t_i}a_0 [ a_1 ,\cdots ,a_i, 1'',a_{j},\cdots, a_{n}] \otimes 1'  [ a_{i+1},\cdots,a_{j-1}] .
\end{split}
\end{equation*}
where 
\begin{equation*}
\begin{split}
t_i &=(|1''|+1)(|a_0|+\cdots + |a_i|+i)+|1'|(|a_0|+|a_1|+\dots+ |a_i|+i+|1''|)\\
&( |1'|+|a_{i+1}|+\cdots +|a_{j-1}|+j-i-1  )(|a_j|+\dots+ |a_n|+n-j+1)\\
&=|1''||1'| +(m+1)(|a_0|+\cdots + |a_i|+i)\\&+( |1'|+|a_{i+1}|+\cdots +|a_{j-1}|+j-i-1  )(|a_j|+\dots+ |a_n|+n-j+1)
\end{split}
\end{equation*}
and for $j=n+1$ and $i=0$ the corresponding terms are respectively 
$$
\pm a_0 [ a_1 ,\cdots ,a_i, 1''] \otimes 1' [ a_{i+1},\cdots,a_{n}].
$$
and
$$
\pm a_0 [ 1'',a_{j},\cdots, a_{n}] \otimes 1' [ a_{1},\cdots, a_{j-1}].
$$
We have 
\begin{equation}\label{eq-cocom}
(D_{Hoch}\otimes 1+ 1\otimes D_{Hoch})h-(-1)^{m+1}hD_{Hoch}=(-1)^m\tau \circ \theta-\theta
\end{equation}
where $\tau: C_*(A,A) ^{\otimes 2}\to C_*(A,A) ^{\otimes 2}$ is given by  $\tau (\alpha_1\otimes \alpha_2)=(-1)^{|\alpha_1||\alpha_2|} \alpha_2\otimes \alpha_1$. To see  this, note that in $(D_{Hoch}\otimes 1+ 1\otimes D_{Hoch})h$ the term corresponding to  the last term of the external part of the Hochchsild differential of the first factor of $\pm a_0 [ a_1 ,\cdots ,a_i, 1''] \otimes 1' [ a_{i+1},\cdots,a_{n}]$  is  $$\pm 1''a_0 [ a_1 ,\cdots ,a_i] \otimes 1' [ a_{i+1},\cdots,a_{n}]= \pm a''_0 [ a_1 ,\cdots ,a_i] \otimes a_0' [ a_{i+1},\cdots,a_{n}]$$ which is precisely $\theta$; and the term corresponding to the first term of the external Hochchsild differential of the first factor of  $\pm a_0 [ 1'',a_{j},\cdots, a_{n}] \otimes 1' [ a_{1},\cdots, a_{j-1}]$ is $\pm a_0  1''[a_{j},\cdots, a_{n}] \otimes 1' [ a_{1},\cdots, a_{j-1}]= \pm a_0  1'[a_{j},\cdots, a_{n}] \otimes 1'' [ a_{1},\cdots, a_{j-1}]=\pm a'_0 [a_{j},\cdots, a_{n}] \otimes a_0'' [ a_{1},\cdots, a_{j-1}]$  which is $(-1)^m\tau \theta$.

To prove that the 7-term (coBV) relation holds, we use the Chas-Sullivan \cite{CS1} idea (see also \cite{Tradler}) in the case of the free loop space adapted to the combinatorial (simplicial) situation. First we identify the Gerstenhaber co-bracket explicitly. Considert the operation
$$
S:= h+(-1)^m\tau\circ h
$$
on the Hochschild complex before the shift of degree of  $m$.
Once proven that $S$ is, up to homotopy, the deviation of $B$ from being a coderivation for $\theta$, the $7$-term homotopy coBV relation is equivalent to the homotopy co-Leibniz identity for $S$.
\subsubsection*{Compatibility of $B$ and $S$:}  We  prove that $S= \theta B -(-1)^m (B \otimes id + id\otimes B) \theta$ up to homotopy. To this end, we prove that $h$ is homotopic to $(\theta B)_2-(-1)^m(B\otimes id)\theta$ and similarly $\tau h \simeq (\theta B)_1-(id \otimes B)\theta $ where  $\theta B= (\theta B)_1 +(\theta B)_2$, with 
\begin{equation*}
\begin{split}
(\theta B)_1(a_0 [a_1,\cdots ,a_n])= \sum_{0\leq i\leq j\leq n}  \sum_{(1)}\pm ( 1'' [a_i,\cdots  , a_{j}])\otimes (1' [a_{j+1},\cdots   ,a_n, a_0 \cdots , a_{i-1}]).
\end{split}
\end{equation*}
and
\begin{equation*}
\begin{split}
(\theta B)_2(a_0 [a_1,\cdots , a_n])= \sum_{0<i<j\leq n} \sum_{(1)} \pm (1'' [a_{j},\cdots  ,a_n ,a_0, a_1, \cdots  , a_{i}])\otimes (1'[a_{i+1},\cdots , a_{j-1}]).
\end{split}
\end{equation*}
The homotopy between $h$ and $(\theta B)_2-(-1)^m(B \otimes id)\theta$ is given by
\begin{equation*}
\begin{split}
H(a_0  [a_1,\cdots , a_n])= \sum_{0 \leq k\leq i < j \leq n+1} \sum_{(a_i)}   (&(-1)^{\nu_{k,i,j}}1[ a_{k+1},\cdots a_{i},1'',a_{j},\cdots, a_{n},a_0,\cdots a_k] )\\ &\otimes (1' [ a_{i+1},\cdots,a_{j-1}] ),
\end{split}
\end{equation*}
where 
\begin{equation*}
\begin{split}
\nu_{k,i,j}&= (|a_{i+1}|+\cdots + |a_{j-1}|+j-i+1)(|a_{j}|+\cdots |a_{n}|+n-j+1) + \\
& (|a_{0}|+\cdots + |a_{k}|+k+1)(|a_{k+1}|+\cdots +|a_i|+|a_j|+\cdots+ |a_{n}|+n-j+i-k+1)+\\
&|1'|(|1''|+|a_{0}|+\cdots +|a_i|+|a_j|+\cdots+ |a_{n}|+n-j+i) +(|1''|+1)( |a_{k+1}|+\cdots +|a_i|+i-k).
\end{split} 
\end{equation*}
In the formulae describing $H$, the sequence $a_j,\cdots, a_{i-1}$ can be empty. In  $D_{Hoch}H+(-1)^{m-2}H(D_{Hoch}\otimes id +id\otimes D_{Hoch})$, the terms corresponding to $k=0$, $k=i$ and $j=n+1$ are respectively $h$, $-(\theta B)_2$ and 
 \begin{equation*}
\begin{split}
(-1)^m(B\otimes 1)\theta(a_0  [a_1,\cdots ,a_n])=(-1)^m\sum \pm (1 [a_{k+1},\cdots, a_i,a'_0,a_1,\cdots, a_{k}] ) \otimes  (a''_0 [a_{i+1}, \cdots , a_n]).
\end{split}
\end{equation*}
Similarly one proves that $\tau h\simeq (\theta B)_1-(-1)^m(id \otimes B)\theta $.
\subsubsection*{Co-Leibniz identity:} The idea of the proof is identical to Lemma 4.6 \cite{CS1}. We prove that up to some homotopy we have
\begin{equation}\label{coleib}
(\theta \otimes id) S=(id\otimes \tau)(S\otimes id)\theta+ (id \otimes S) \theta
\end{equation}
At the chain level, we have
$$
(\theta \otimes id)h=(id\otimes \tau)(h\otimes id)\theta + (id \otimes h)\theta,
$$
so to prove (\ref{coleib})  we should prove that up to some homotopy
\begin{equation*}
(\theta \otimes id)\tau h=(id\otimes \tau)(\tau h\otimes id)\theta + (id\otimes  \tau h)\theta.
\end{equation*}
The homotopy is given by $G: C^*(A)\rightarrow (C^*(A))^{\otimes 3} $
\begin{equation*}
\begin{split}
G(a_0 [a_1, \cdots ,a_n])=\sum_{0\leq l < i \leq j<k} \sum_{(1),(1)} & \pm( 1''_1[a_{l+1},\cdots, a_{i}])  \otimes (1''_2  [a_{j+1},\cdots, a_{k}])\\ &\otimes a_0[a_1,\cdots, a_l,1'_1 ,a_{i+1},\dots , a_{j},1'_2,a_{k+1} \cdots a_n,a_0,\cdots, a_{l-1}],
\end{split}
\end{equation*}
that is 
\begin{equation*}
\begin{split}
&GD_{Hoch}+(-1)^{m-2}(D_{Hoch}\otimes id \otimes id+id\otimes D_{Hoch} \otimes id+id\otimes id \otimes D_{Hoch})G\\&= (\theta \otimes id)\tau h-(id\otimes \tau)(\tau h\otimes id)\theta -(id\otimes  \tau h)\theta.
\end{split}
\end{equation*}
The signs in $G$ are determined using Koszul sign rule just like the previous examples

As for the last part of the theorem, we have already proved that  $\odot$ corresponds to the cup. It only remains to prove that underlying Gerstenharber bracket of the  BV-structure
of $ (C^*(A,A),\cup)$ is the standard one. To that end, it suffices to prove that $\widetilde{} :  (C^*(A,A)[m], \cup) \to (C^*(A,A^\vee)[m],\odot)$  (see (\ref{identif-eq})) sends  the homotopy of the commutativity  $\circ$ of   
$\cup$ to the homotopy of the commutativity $(-\otimes -)h$ of $\odot$: Let $x\in \Hom ((sA)^{\otimes p},A)$ and $y\in \Hom ((sA)^{\otimes q},A)$.  Then the degrees of $\tilde{x} $ and $\tilde{y}$ as element of $\Hom (C_*(A,A),\kk)[m]$ are respectively $|x|$ and $|y|$. Similarly to the definition of $\odot$, the homotopy for the commutativity of $\tilde{x}$ and $\tilde{y}$
is given   $\Psi (\tilde{x},\tilde{y})=(-1)^{(m-1)|\tilde{y}|}\mu (\tilde{x}\otimes \tilde{y})h$ where $\mu:\kk\otimes \kk\to \kk$ is the product of the ground ring:
\begin{equation*}
\begin{split}
&\Psi (\tilde{x},\tilde{y})(a_0,a_1,\cdots,a_{p+q-1})\\&=\sum_{i,(1)} \pm  \tilde{x}( a_0, a_1,\cdots, a_i,1'' , a_{i+q+1},\cdots a_{p+q-1})\tilde{y} (1', a_{i+1},\cdots a_{i+q})
\end{split}
\end{equation*}
and on the other hand (see (\ref{def-circ})):
\begin{equation*}
\begin{split}
&\widetilde{x\circ y} (a_0,a_1,\cdots a_{p+q-1})=(-1)^{(|a_0|+1)( |x|+|y|-1)}  \langle a_0,(x\circ y)(a_1,\cdots a_{p+q-1})\rangle \\&= \sum _i \pm \langle a_0,x(a_1,\cdots, a_i, y(a_{i+1},\cdots y_{i+q}), a_{i+q+1},\cdots a_{p+q-1})\rangle \\
&=\sum _i \pm \langle a_0,x(a_1,\cdots, a_i, \sum_{(1)} 1''\langle 1', y(a_{i+1},\cdots a_{i+q})\rangle , a_{i+q+1},\cdots a_{p+q-1})\rangle\\
& =\sum _{i,(1)} \pm \langle a_0,x(a_1,\cdots, a_i,1'' , a_{i+q+1},\cdots a_{p+q-1})\rangle \langle 1', y(a_{i+1},\cdots a_{i+q})\rangle\\
& =\sum _{i,(1)} (-1)^{|(|y|+1)((|a_0|+|1'|+|a_1|+\cdots +|a_i|+i)+|1'|} \tilde{x}( a_0,\cdots, a_i,1'' , a_{i+q+1},\cdots a_{p+q-1})\tilde{y} (1', a_{i+1},\cdots a_{i+q})\\
\end{split}
\end{equation*}
A comparison of the signs by using the identity $ |y|+|1'|+|a_{i+1}|\cdots +|a_{i+q|}|=m$, finishes the proof.
\end{proof}

\begin{remark}
By  F\'elix-Thomas \cite{FTBV} theorem, this cup product on $HH^*(A,A^\vee)$ provides an algebraic model for the Chas-Sullivan product on $H_*(LM)$ the homology of the free loop space of closed oriented manifold $M$. Here one must work over a field of characteristic zero and for $A$ one can take the closed (commutative) Frobenius algebra provided by Lambreschts-Stanley result \cite{LamStan} on the existence of an algebraic model with Poincar\'e duality for the rational singular cochain algebra of a closed oriented maniflold.
\end{remark}

\section{BV structure on Hochschild homology}

Although there is no action of the chains of  the moduli space of Riemann surfaces on  the Hochschild complex of an open Frobenius algebra, some parts of such action in the case of  closed Frobenius algebras can be formulated using the product and coproduct of the underlying algebra (see also \cite{WahWest}). Therefore one can write down various operations in the Hochschild homology of an open Frobenius algebra, but the desired identities have to be proved directly by giving explicitly the required homotopies. Here is one example.

\begin{theorem}\label{thm-co-op}
For $A$ a symmetric open Frobenius algebra,  the shifted Hochschild homology $HH_*(A,A)[m]$ can be naturally equipped with a BV-structure whose BV-operator is  the Connes operator and the product at the chain level (before the shift of degree) is given by  degree $m-1$
\begin{equation}\label{prod}
a_0[a_1,\cdots,a_p]\bullet b_0[b_1,\cdots,b_q]=\begin{cases} 0 & \text{ if } p>0 \\ (-1)^{|a'_0||a''_0|}a''_0a'_0b_0[b_1,\cdots, b_q].  \end{cases}
\end{equation}
\end{theorem}
\begin{proof}
Proposition \ref{pro-frob} (4)implies that $\bullet$ is a chain map (of degree $m$). The product $\bullet$ is strictly associative. We only have to check this for $x=a[\quad]$, $y=b[\quad]$ and $z=c[c_1,\cdots,c_n]$. Using Proposition \ref{pro-frob}, we have
\begin{equation*}\begin{split}
&(x \bullet y)\bullet z= \sum (-1)^{|a'||a''|+|(a''a'b)''||(a''a'b)'|}(a''a'b)''(a''a'b)'c_0[c_1,\cdots,c_n]=\\
 & =(-1)^{|a'||a''|+m(|a'|+|a''|)+ |b''| (|b'|+|a'|+|a''|) }\sum b''a''a'b'c_0[c_1,\cdots,c_n]\\
 & =(-1)^{|a'||a''|+m(|a'|+|a''|)+ |b'' | |b'|+|b'' |(|a'|+|a''|)+|b'' | (|a'|+|a''|) }\sum a''a' b''b'c_0[c_1,\cdots,c_n]\\
 &= (-1)^{|a'||a''|+m(|a'|+|a''|)+ |b'' | |b'|) }\sum a''a' b''b'c_0[c_1,\cdots,c_n]\\
  &= (-1)^{|a'||a''|+m|a|+m^2+ |b'' | |b'|}\sum a''a' b''b'c_0[c_1,\cdots,c_n]\\
\end{split}
\end{equation*}
On the other hand
\begin{equation*}\begin{split}
x\bullet(y\bullet z)&= (-1)^{|a'||a''|+ |b'' | |b'| }\sum a''a' b''b'c_0[c_1,\cdots,c_n]\\
&= (-1)^{|a'||a''|+ |b'' | |b'| }\sum a''a' b''b'c_0[c_1,\cdots,c_n]=(-1)^{m|x|+m}(x\bullet y)\bullet z
\end{split}
\end{equation*}
Next we prove that the product is commutative up to homotopy. Indeed the homotopy for $x=a_0[a_1,\cdots,a_p]$ and $y= b_0[b_1,\cdots,b_q]$ is given by
\begin{equation*}
K(x,y)= \sum_{(a_0)}(-1)^{(|a'_0|+1)(|a''_0|+|a_1|+\cdots +|a_p| +p)} a''_0[ a_1 , \dots , a_p ,  a_0'b_0 , b_1 , \dots , b_q],\\
\end{equation*}
that is
\begin{equation*}
D_{Hoch}K -(-1)^{m-1} K(D_{Hoch}\otimes 1+1\otimes D_{Hoch})=x\bullet y-(-1)^{|y||x|+m}y\bullet x .
\end{equation*}
Note that  $\deg K= m-1$.
It is instructive to verify the case  $p=0$. The most nontrivial case of cancellation in  $D_{Hoch}K -(-1)^{m-1} K(D_{Hoch}\otimes 1+1\otimes D_{Hoch})$ follows from
the identity 
\begin{equation*}
\begin{split}
&\sum(-1)^{(|a'_0|+1)|a''_0|+(|b_q|+1)(|a_0''|+|a_0'|+|b_0|+\cdots +|b_{q-1}|+q) } b_qa''_0[a'_0b_0,\cdots,b_{q-1}]\\
&= (-1)^{m-1}\sum(-1)^{|a_0|+|a_0''|(|a_0'|+1)+(|b_q|+1)(|b_0|+\cdots +|b_{q-1}|+q-1) } a''_0[a_0'b_qb_0,\cdots,b_{q-1}]
\end{split}
\end{equation*}
whose proof is as follows:
\begin{equation*}
\begin{split}
&\sum(-1)^{(|a'_0|+1)|a''_0|+(|b_q|+1)(|a_0''|+|a_0'|+|b_0|+\cdots +|b_{q-1}|+q) } b_qa''_0[a'_0b_0,\cdots,b_{q-1}]=\\
&\sum(-1)^{(|1'|+1)(|1''|+|a_0|)+(|b_q|+1)(|1''|+|1'|+|a_0|+|b_0|+\cdots +|b_{q-1}|+q) } b_q1''a_0[1'b_0,\cdots,b_{q-1}]=\\
&\sum(-1)^{(|1'|+1)(|1''|+|a_0|)+(|b_q|+1)(m+|a_0|+|b_0|+\cdots +|b_{q-1}|+q) } b_q1''a_0[1'b_0,\cdots,b_{q-1}]=\\
&\sum(-1)^{m+|1''||1'|+(|1''|+1)(|1'|+|a_0|)+(|b_q|+1)(m+|a_0|+|b_0|+\cdots +|b_{q-1}|+q) } b_q1'a_0[1''b_0,\cdots,b_{q-1}]=\\
&\sum(-1)^{m+m|b_q|+|b_q''|(|b_q'|+|b_p|)+(|b_q''|+1)(|b_q'|+|b_p|+|a_0|)+(|b_p|+1)(m+|a_0|+|b_0|+\cdots +|b_{q-1}|+q) } b'_qa_0[b_q''b_0,\cdots,b_{q-1}]=\\
&\sum(-1)^{m+m|b_q|+(|1''|+ |b_q|)(|1'|+|b_q|)+(|1''|+|b_q|+1)(|1'|+|b_q|+|a_0|)+(|b_q|+1)(m+|a_0|+\cdots |b_{q-1}|+q) } 1'a_0[1''b_qb_0,\cdots,b_{q-1}]\\
&=\sum(-1)^{m+|1''||1'|+|b_p|+(|1''|+|b_p|+1)(|1'|+|b_p|+|a_0|)+(|b_p|+1)(m+|a_0|+\cdots +|b_{p-1}|+q) } 1'a_0[1''b_qb_0,\cdots,b_{q-1}]\\
&=\sum(-1)^{|b_q|+(|1'|+|b_q|+1)(|1''|+|b_q|+|a_0|)+(|b_q|+1)(m+|a_0|+|b_0|+\cdots +|b_{q-1}|+q) } 1''a_0[1'b_qb_0,\cdots,b_{q-1}]\\
&=\sum(-1)^{|b_q|+(|a_0'|+|b_q|+1)(|a_0''|+|b_q|)+(|b_q|+1)(m+|a_0|+|b_0|+\cdots +|b_{q-1}|+q) } a''_0[a_0'b_qb_0,\cdots,b_{q-1}]\\
&=\sum(-1)^{m-1+|a_0|+|a_0''|(|a_0'|+1)+(|b_q|+1)(|b_0|+\cdots +|b_{q-1}|+q-1) } a''_0[a_0'b_qb_0,\cdots,b_{q-1}]\\
\end{split}
\end{equation*}

Let us examine the case of  $p=q=0$.  For  $x=a[\quad]$ and $y=b[\quad ]$: for the external differential we have  
\begin{equation*}
\begin{split}
d_1K(x,y)&= \sum (-1)^{|a''|+(|a'|+1)|a''|}a''a'b[\quad ] + \sum (-1)^{|a''| (|a'|+|b|+1)+1+(|a'|+1)|a''|}a'b a''[\quad ]  \\&=x\bullet y-  \sum (-1)^{|a''| |b|}a'b a''[\quad ]=x\bullet y-  \sum (-1)^{(|1''|+|a|) |b|}1'b 1''a[\quad ]\\&
=x\bullet y-  \sum (-1)^{(|1'|+|a|) |b|+|1'||1''|+m}1''b 1'a[\quad ]\\&=x\bullet y-  \sum (-1)^{(|b'|+|a|) |b|+|b'|(|b''|+|b|)+m}b''b'a[\quad ]\\
&=x \bullet y-  (-1)^{|a||b|+m}\sum (-1)^{(|b'||b'|}b''b'a[\quad ]=x\bullet y-  (-1)^{|x||y|+m}y\bullet x
\end{split}
\end{equation*}
As for the internal differential $d_0$, we have $d_0K(x,y)=(-1)^{m-1}K( d_0\otimes 1+ 1\otimes d_0)(x\otimes y)$, therefore 
$$
D_{Hoch}K -(-1)^{m-1} K(D_{Hoch} \otimes1 +1\otimes D_{Hoch})(x\bullet y) =x\bullet y-(-1)^{|x||y|+m}y\bullet x. 
$$
The Gerstenhaber bracket is naturally defined to be 
\begin{equation}\label{eq-gers}
\{x,y\}:=  K(x,y) +(-1)^{|x||y|+m}K(y,x).
\end{equation}
Next we prove that  the identity
\begin{equation}\label{eq-devia}
\{x,y\}= B(x\bullet y)-(-1)^m(B x \bullet y-(1)^{|x|} x\bullet B y  )
\end{equation}
holds up to homotopy.  First note that $B x\bullet y =0$ for all $x$ and $y$. A homotopy between all the remaining three terms is given
by $H+(-1)^{1+m} K(1\otimes B )$ where
  
\begin{equation*}
\begin{split}
H(x,y)=\sum_{(a), 1 \leq k\leq q +1}(-1)^{\alpha_k} 1[b_k,\cdots, b_q,a_0'',a_1,\cdots a_p,a'_0b_0,b_1,\cdots b_{k-1}],
\end{split}
\end{equation*}
with
\begin{equation*}
\begin{split}
\alpha_k&=(|a_0'|+1) (|a_0''|+|a_1|+\cdots +|a_p|+p)\\&+ (|b_k|+\cdots | b_q|+q-k-1)(|a_0'|+|a_0''|+|a_1|+\cdots +|a_p|+|b_0|+|b_1|+\cdots |b_{k-1}| +k+p+1).
\end{split}
\end{equation*}

First notice that  $[D_{Hoch}, K\circ (1\otimes B)](x\otimes y)$ is exactly  $x\bullet B y$. To analyse the rest we have to consider two cases:
\begin{itemize}
 \item Case $p\geq  1$: In this case  $B (x\bullet y)=0$. In computing $[D_{Hoch}, H] (x\otimes y)$ only two term survives. Those are the ones corresponding to $k=1$ and $k=q+1$.
 The one corresponding to $k=q+1$ gives us exactly  $K(x,y)$ and  for $k=1$ we obtain  $(-1)^{m+|x||y|} K(y,x)$. This terms is produced when we compute the last
 term of the external part of $D_{Hoch} H$:
 \begin{equation*}
\begin{split}
&(-1)^{ (|a_0'|+1) (|a_0''|+|a_1|+\cdots +|a_p|+p)+ (|b_1|+\cdots | b_q|+q)(|a_0'|+|a_0''|+|a_1|+\cdots +|a_p|+|b_0|+p)}\\
&(-1)^{(|a'_0|+|b_0|+1)(|b_1|+\cdots+|b_q|+|a_0''|+|a_1|+\cdots p+1+q)+1}a'_0b_0[b_1,\cdots,b_q,a''_0,a_1\cdots, a_p]=
\\& (-1)^{|a_0'|+(|b_0|+|b_1|+\cdots+|b_q|+q)(|a_0''|+|a_1|+\cdots+|a_p|+p+1)}a'_0b_0[b_1,\cdots,b_q,a''_0,a_1\cdots, a_p]=
\\&  (-1)^{|1'|+(|b_0|+|b_1|+\cdots+|b_q|+q)(|1''|+|a_0|+|a_1|+\cdots+|a_p|+p+1)}1'b_0[b_1,\cdots,b_q,1''a_0,a_1\cdots, a_p]=
\\&  (-1)^{m+|1'||1''|+|1''|+(|b_0|+|b_1|+\cdots+|b_q|+q)(|1'|+|a_0|+|a_1|+\cdots+|a_p|+p+1)}1''b_0[b_1,\cdots,b_q,1'a_0,a_1\cdots, a_p]=
\\&  (-1)^{m+(|b_0'|+1)(|b_0''|+|b_0|)+(|b_0|+|b_1|+\cdots+|b_q|+q)(|b_0'|+|a_0|+|a_1|+\cdots+|a_p|+p+1)}b''_0[b_1,\cdots,b_q,b_0'a_0,a_1\cdots, a_p]=
\\&  (-1)^{m+ |x||y|+(|b_0'|+1)(|b_0''|+|b_1|+\cdots+|b_q|+q)}b''_0[b_1,\cdots,b_q,b_0'a_0,a_1\cdots, a_p]=(-1)^{m+|x||y|} K(y,x),
\end{split}
\end{equation*}
therefore $[D_{Hoch}, H] (x\otimes y)= \{x,y\}$.
\item  Case $p=0$: In this case, comparing to the previous case, an extra term in $[D_{Hoch}, H] (x\otimes y)$ shows up. This is  the term where $a_0''$ and $a_0'b_0$ are multiplied.
This  is precisely
\begin{equation*}
\begin{split}
B(x\bullet y)=\sum_{i=1}^{n} \sum_{(a)}\pm1[b_i,\cdots, b_q, a''_0a'_0b_0,b_1,\cdots, b_{i-1} ]
\end{split}
\end{equation*}
which is not necessarily zero if $p=0$.
\end{itemize}
Finally we prove the that Leibniz identity (before  the shift of the  grading)
$$
\{x, y\bullet z\}= \{x, y\}\bullet z+ (-1)^{(m-1+|x|)|y|} y\bullet \{ x,z\}
$$  
holds up to homotopy. We prove that it in fact it holds strictly.
First note that if  $y\in \oplus_{n>0} ( A\otimes (sA)^{\otimes n})$ then all the terms vanish. Therefore we suppose that  $y=b[\quad ]$. Since $\{x, y\} \in  \oplus_{n>0} ( A\otimes (sA)^{\otimes n})$,
it suffices to prove that $\{x, y\bullet z\}=  (-1)^{(m-1+|x|)|y|} y\bullet \{ x,z\}$, and to that end we check the follow identities
$$
K(x, y\bullet z)= (-1)^{(m-1+|x|)|y|} y\bullet K(x,z)
$$
and 
$$
K( y\bullet z,x)=(-1)^{(m-1+|x|)|y|+ |x||y|+m} y\bullet K(z,x).
$$
We prove the first identity, the second one is similar.  For $x=a_0[a_1,\cdots,a_p]$ and $z=c_0[c_1,\cdots c_q]$, we have
\begin{equation*}
\begin{split}
K(x, y\bullet z)= (-1)^{|b''||b'|+ (|a_0'|+1)(|a_0''|+|a_1|\cdots+ |a_p|+p)}a_0''[a_1,\dots, a_p,a_0'b''b'c_0,c_1\cdots, c_q]
\end{split}
\end{equation*}
and 
\begin{equation*}
\begin{split}
y\bullet  K(x, z)= (-1)^{|b''||b'|+ (|a_0'|+1)(|a_0''|+|a_1|\cdots+ |a_p|+p)}b''b' a_0''[a_1,\dots, a_p,a_0'c_0,c_1\cdots, c_q]
\end{split}
\end{equation*}
The  claimed equality is proved as follows:
\begin{equation*}
\begin{split}
&(-1)^{|b''||b'|+ (|a_0'|+1)(|a_0''|+|a_1|\cdots+ |a_p|+p)}a_0''[a_1,\dots ,a_0'b''b'c_0,\cdots, c_q]\\&= (-1)^{m|a_0|+|b''||b'|+ (|a_0|+|1'|+1)(|1''|+|a_1|\cdots+ |a_p|+p)}1''[a_1,\dots ,a_01'b''b'c_0,\cdots, c_q]\\&
(-1)^{m|a_0|+|b''||b'|+ |1'||1''|+m+(|a_0|+|1''|+1)(|1'|+|a_1|\cdots+ |a_p|+p)}1'[a_1,\dots ,a_01''b''b'c_0,\cdots, c_q]\\
&=(-1)^{m|a_0|+|b''||b'|+ |(b''b')'|(|(b''b')''|+|b''b'|)+m+(|a_0|+|(b''b')''|+|b''b'|+1)(|(b''b')'|+|a_1|\cdots+ |a_p|+p)}\\ & \qquad  \qquad  \qquad  \qquad  \qquad (b''b')'[a_1,\dots ,a_0(b''b')''c_0,\cdots, c_q]=\\
&=(-1)^{m|a_0|+m(m+|b|)+|b''||b'|+ (|b''b'|+|1'|)(|1''|+|b''b'|)+m+(|a_0|+|1''|+|b''b'|+1)(|1'|+|b''b'|+|a_1|\cdots+ |a_p|+p)}\\ & \qquad  \qquad  \qquad  \qquad  \qquad b''b'1'[a_1,\dots ,a_01''c_0,\cdots, c_q]\\
&=(-1)^{m|a_0|+m(m+|b|)+|b''||b'|+ (m+|b|+|1'|)(|1''|+m+|b|)+m+(|a_0|+|1''|+|b|+m+1)(|1'|+m+|b|+|a_1|\cdots+ |a_p|+p)}\\ & \qquad  \qquad  \qquad  \qquad  \qquad b''b'1'[a_1,\dots ,a_01''c_0,\cdots, c_q]\\
&=(-1)^{m|a_0|+|b''||b'|+ |1'||1''|+m+|b|+m+(|a_0|+|1''|+|b|+m+1)(|1'|+m+|b|+|a_1|\cdots+ |a_p|+p)} b''b'1'[a_1,\dots ,a_01''c_0,\cdots, c_q]\\
&=(-1)^{m|a_0|+|b''||b'|+ |b|+m+(|a_0|+|1'|+|b|+m+1)(|1''|+m+|b|+|a_1|\cdots+ |a_p|+p)} b''b'1''[a_1,\dots ,a_01'c_0,\cdots, c_q]\\
&=(-1)^{|b''||b'|+ |b|+m+(|a'_0|+|b|+m+1)(m+|b|+|a_0''|+|a_1|\cdots+ |a_p|+p)} b''b'a_0''[a_1,\dots ,a_0'c_0,\cdots, c_q]\\
&=(-1)^{m|x|+(m-1+|x|)|y|+|b''||b'|+ (|a_0'|+1)(|a_0''|+|a_1|\cdots+ |a_p|+p)}b''b' a_0''[a_1,\dots, a_p,a_0'c_0,c_1\cdots, c_q].
\end{split}
\end{equation*}
\end{proof}
\section{Frobenius compatibility of the product and coproduct} 
As we mentioned  previously we are inspired by the algebraic structures of the homology of free loop spaces.  Cohen-Godin \cite{CG} result holds even for the manifolds which are not closed. The difference with the closed case would be that there won't be a counit for the underlying algebra structure. The coalgebra structure generically has no counit, otherwise the homology of the free loop space would have the homotopy type of a finite dimensional manifold which is not true except for very special kind of aspherical manifolds. Therefore it is natural to expect a Frobenius compatibility condition (Definition \ref{defopfrob}, (3)) between the product and coproduct.

\begin{theorem} Let $A$ be a symmetric open Frobenius algebra. The product  $\bullet$  \eqref{prod} and coproduct $\theta$ \eqref{eq-cup-frob}  on $HH_*(A,A)[m]$ satisfy the Frobenius compatibility conditions, Definition \ref{defopfrob}, (3).

\end{theorem}

\begin{proof}
We have to prove that  $\theta$ is a map of (degree $m$) left and right $HH_*(A,A)$-modules, that is $(-1)^m (1\otimes \bullet ) \circ (\theta\otimes 1)= \theta \circ \bullet$ at homology level.
 First we  consider the right $HH_*(A)$-module structure.  Let  $x= a_0[a_1,\cdots, a_m]$ and  $y= b_0[b_1,\cdots, b_n] \in C_*(A,A)$.  If $m\geq 1$, then  $\theta(x\bullet y)= 0$, and 
\begin{equation*}\begin{split}
\theta x \bullet y &=(\sum_{(a_0), 1\leq k\leq m} (-1)^{\sigma_k |a_0'|}a''_0[a_1,\cdots,a_k] \otimes a'_0[a_{k+1},\cdots, a_m ])\bullet y\\
&=(\sum_{(a_0), 1\leq k\leq m}  (-1)^{\sigma_k |a_0'|} a''_0[a_1,\cdots,a_k] \otimes ( a'_0[a_{k+1},\cdots, a_m ] \bullet y)\\
&= \sum_{(a_0),(a_0')}(-1)^{\sigma_m |a_0'|+|(a_0')'||(a_0')''|}a''_0[a_1,\cdots,a_m] \otimes (a'_0)''(a'_0)'b_0[b_1,\cdots,b_n]\\
& =\sum_{(a_0), (1)} (-1)^{m+\sigma_m |a_0'|+|a_0'||(a_0'')'|} (a''_0)''[a_1,\cdots,a_m] \otimes (a''_0)'a'_0 b_0[b_1,\cdots,b_n]\\
& =\sum_{(a_0), (1)} (-1)^{m+\sigma_m |a_0'|+|a_0'|(|1''| +|a''_0|)} 1''a''_0[a_1,\cdots,a_m] \otimes 1'a'_0 b_0[b_1,\cdots,b_n].
\end{split}
\end{equation*}
So we have to prove the latter is homotopic to zero. The homotopy is given by 
\begin{equation*}\begin{split}
H(x,y)&=\sum_{(a_0), (1)} (-1)^{m+\sigma_m (|a_0'|+1)+|a_0'|(|1''| +|a''_0|)+|1'|} 1''a''_0[a_1,\cdots,a_m] \otimes 1'[ a'_0 b_0,b_1,\cdots,b_n]
\end{split}
\end{equation*}
Two non-trivial cancellations occur in computing  $D_{Hoch}H-(-1)^{2m} H(D_{Hoch}\otimes 1+I\otimes D_{Hoch}))$  as consequences of the follow identities (we omit the signs for the sake of simplicity):
\begin{equation*}\begin{split}
&a_m1''a''_0[a_1,\cdots,a_{m-1}] \otimes 1'a'_0 b_0[b_1,\cdots,b_n]=a_m1'a''_0[a_1,\cdots,a_{m-1}] \otimes 1''a'_0 b_0[b_1,\cdots,b_n]\\
&=a_m'a''_0[a_1,\cdots,a_{m-1}] \otimes a_m''a'_0 b_0[b_1,\cdots,b_n]=1'a''_0[a_1,\cdots,a_{m-1}] \otimes 1'' a_ma'_0 b_0[b_1,\cdots,b_n]\\
&=1'(a_ma_0)''[a_1,\cdots,a_{m-1}] \otimes 1'' (a_ma_0)' b_0[b_1,\cdots,b_n]\\
\end{split}
\end{equation*}
and
\begin{equation*}\begin{split}
 &1''a''_0[a_1,\cdots,a_{m}] \otimes b_n1'[a'_0 b_0,b_1,\cdots,b_{n-1}] =  b_n''a''_0[a_1,\cdots,a_{m}] \otimes b_n'[a'_0 b_0,b_1,\cdots,b_{n-1}] \\
 &= 1'' b_na''_0[a_1,\cdots,a_{m}] \otimes 1'[a'_0 b_0,b_1,\cdots,b_{n-1}] =1'' b_n1''[a_1,\cdots,a_{m}] \otimes 1'[a_01' b_0,b_1,\cdots,b_{n-1}] 
 \\ &=1''b_n1'[a_1,\cdots,a_{m}] \otimes 1'[a_01'' b_0,b_1,\cdots,b_{n-1}]  =1'' b_n'[a_1,\cdots,a_{m}] \otimes 1'[a_0b_n'' b_0,b_1,\cdots,b_{n-1}] \\
 &=1'' 1'[a_1,\cdots,a_{m}] \otimes 1'[a_01''b_n b_0,b_1,\cdots,b_{n-1}] =1'' 1''[a_1,\cdots,a_{m}] \otimes 1'[a_01'b_n b_0,b_1,\cdots,b_{n-1}] \\
 &=1'' a_0''[a_1,\cdots,a_{m}] \otimes 1'[a_0'b_n b_0,b_1,\cdots,b_{n-1}] 
 \end{split}
\end{equation*}
 
 If $m=0$, then for $x=a[~]$
\begin{equation*}
 \begin{split}
 & \theta(x\bullet y)=\theta(\sum_{(a)} (-1)^{|a'||a''|} a''a'b_0[b_1,\cdots,b_n])\\
		  &=\sum_{\substack{(a),(a''a'b_0)\\ 0\leq k\leq n}}(-1)^{|a'||a''|+ |(a''a'b_0)'|(|(a''a'b_0)''|+|b_1|+\cdots |b_k|+k) )}  (a''a'b_0)''[b_1,\cdots,b_k] \otimes (a''a'b_0)'[b_{k+1},\cdots,b_n]\\
		  &=\sum_{\substack{(a),(1)\\ 0\leq k\leq n}} (-1)^{|a'||a''|+ m|a''|+ (|(a'b_0)'|+|a''|)( |(a'b_0)''|+|b_1|+\cdots |b_k|+k) }  (a'b_0)''[b_1,\cdots,b_k] \otimes a''(a'b_0)'[b_{k+1},\cdots,b_n]\\		 
		  &=\sum_{\substack{(a),(1)\\ 0\leq k\leq n}} (-1)^{|a'||a''|+ m|a''|+ (|1'|+|a''|) (|1''| +|a'|+|b_0|+|b_1|+\cdots |b_k|+k)}  1'' a'b_0[b_1,\cdots,b_k] \otimes a''1'[b_{k+1},\cdots,b_n]\\
 \end{split}
\end{equation*}
On the other hand,
\begin{equation*}
 \begin{split}
& [(-1)^m (1\otimes \bullet ) \circ (\theta\otimes 1)](x\otimes y)=\sum_{(a)}(-1)^{m+m|a''|+|a'||a''|}a''[~]\otimes (a'[~]\bullet y) \\
&= \sum_{(a)}(-1)^{m+|a''|(|a|+|a''|)}a''[~]\otimes (a'[~]\bullet y) \\
 &=\sum_{(a),(a'')} (-1)^{m+|a''|(|a|+|a''|)+|(a')'||(a')''|} a''[~]\otimes (a')''(a')'b_0[b_1,\cdots ,b_n]\\
	  &=\sum_{(a),(a'')} (-1)^{m|a'|+|(a'')''|(|a|+|(a'')''|)+|a'||(a'')'|} (a'')''[~]\otimes (a'')'a'b_0[b_1,\cdots ,b_n]\\
	  &=\sum_{(a),(a'')} (-1)^{m|a''|+m|a'|+|1''|(|a|+|1''|)+|a'|(|a''|+|1'|)} 1''[~]\otimes a'' 1'a'b_0[b_1,\cdots ,b_n]\\
	    &=\sum_{(a),(a'')} (-1)^{m|a'|+m|a''|+|1'|(|a|+|1'|)+|a'|(|a''|+|1''|)+|1'||1''|+m} 1'[~]\otimes a'' 1''a'b_0[b_1,\cdots ,b_n]\\
	     &=\sum_{(a),(a'')} (-1)^{m(|a'|+|a''|)+|1'|(|a|+m)+|a'|(|a''|+|1''|)+m} 1'[~]\otimes a'' 1''a'b_0[b_1,\cdots ,b_n]\\
	     &=\sum_{(a),(a'')} (-1)^{m(|a'|+|a''|)+|(a'b_0)'|(|a|+m)+|a'|(|a''|+|(a'b_0)''|+ |a'b_0|)+m} (a'b_0)'[~]\otimes a'' (a'b_0)''[b_1,\cdots ,b_n]\\
 &=\sum_{(a),(a'')} (-1)^{m(|a'|+|a''|)+m|a'b_0|+(|a'b_0|+|1'|)(|a|+m)+|a'|(|a''|+|1''|+ |a'b_0|)+m } a'b_0 1'[~]\otimes a'' 1''[b_1,\cdots ,b_n]\\
  &=\sum_{(a),(a'')} (-1)^{m(|a'|+|a''|)+m|a'b_0|+(|a'b_0|+|1''|)(|a|+m)+|a'|(|a''|+|1'|+ |a'b_0|)+|1'||1''|} a'b_0 1'[~]\otimes a'' 1''[b_1,\cdots ,b_n]\\
    &=\sum_{(a),(a'')} (-1)^{m|b|+m|a'|+m|a''|+|1'|1''|+|a''|b|+|1''||a''|} a'b_0 1''[~]\otimes a'' 1'[b_1,\cdots ,b_n]\\
        \end{split}
\end{equation*}
The homotopy between $(-1)^m (1\otimes \bullet ) \circ (\theta\otimes 1) (x\otimes y)$  and  $\theta \circ \bullet (x\otimes y)$ is given by

\begin{equation*}
 G(x,y):=\sum_{\substack{(1),(1)\\ 0\leq k\leq n}} (-1)^{|1''|+|a'||a''|+ m|a''|+ (|1'|+|a''|) (|1''| +|a'|+|b_0|+\cdots |b_k|+k)}1''[a'b_0,b_1,\cdots,b_k] \otimes a''1'[b_{k+1},\cdots,b_n],
\end{equation*}

The left $HH_*(A,A)$-module condition $(-1)^m \theta\circ \bullet=(\bullet \otimes 1 )(1\otimes \theta)$ actually  holds at the chain level.
The only nontrivial case is when $m=0$, otherwise both $x\bullet \theta(y)$ and $\theta(x\bullet y)$ are zero.
For $x=a[~]$ we have, 
\begin{equation*}
 \begin{split}
 (\bullet \otimes 1 )(1\otimes \theta)(x\otimes y)& =(-1)^{m|x|}x\bullet \theta(y)\\&= \sum_{\substack{(b_0)\\0\leq k\leq n}} (-1)^{m|a|+|b_0'|(|b_0''|+\cdots +|b_k|+k)}(x\bullet b_0''[b_1,\cdots ,b_k])\otimes b'_0[b_{k+1},\cdots , b_n]\\
		  &= \sum_{\substack{(b_0),(a_0)\\0\leq k\leq n}}(-1)^ {m|a|+|a''||a'|+|b_0'|(|b_0''|+\cdots +|b_k|+k)} a''a'b_0''[b_1,\cdots ,b_k]\otimes b'_0[b_{k+1},\cdots , b_n].
 \end{split}
\end{equation*}
On the other hand,
\begin{equation*}
 \begin{split}
& (-1)^m (\theta\circ \bullet) (x\otimes y)= (-1)^m\theta(x\bullet y)=\theta(\sum_{\substack{(a)}} (-1)^{m+|a'||a''|} a''a'b_0[b_1,\cdots ,b_n])\\&=\theta(\sum_{\substack{(a)}} (-1)^{m+|a'||a''|+|b_0|(m+|a|)} b_0 a''a'[b_1,\cdots ,b_n])\\
 &= \sum(-1)^{m+|a'||a''|+|b_0|(|a''|+|a'|)+|(b_0a''a')'|( |(b_0a''a')''|+\cdots +|b_k|+k)}  (b_0a''a')''[b_1,\cdots ,b_k]\otimes (b_0a''a')'[b_{k+1},\cdots , b_n]\\
&= \sum(-1)^{m+|a'||a''|+|b_0|(|a''|+|a'|)+|b_0'|(|a''|+|a'|+ |b_0''|+|b_1|+\cdots +|b_k|+k)}  b''_0a''a'[b_1,\cdots ,b_k]\otimes b_0'[b_{k+1},\cdots , b_n]\\
&= \sum(-1)^{m+|a'||a''|+(|b_0''|+|b_0|)(|a''|+|a'|)+|b_0'|(|a''|+|a'|+ |b_0''|+|b_1|+\cdots +|b_k|+k)}  a''a' b''_0[b_1,\cdots ,b_k]\otimes b_0'[b_{k+1},\cdots , b_n]\\
&= \sum(-1)^{m+|a'||a''|+(|b_0''|+|b_0|+|b_0'|)(|a''|+|a'|)+|b_0'|(|b_0''|+|b_1|+\cdots +|b_k|+k)}  a''a' b''_0[b_1,\cdots ,b_k]\otimes b_0'[b_{k+1},\cdots , b_n]\\
&= \sum(-1)^{m+|a'||a''|+m(|a|+m)+|b_0'|(|b_0''|+|b_1|+\cdots +|b_k|+k)}  a''a' b''_0[b_1,\cdots ,b_k]\otimes b_0'[b_{k+1},\cdots , b_n]\\
&= \sum(-1)^{|a'||a''|+m|a|+|b_0'|(|b_0''|+|b_1|+\cdots +|b_k|+k)}  a''a' b''_0[b_1,\cdots ,b_k]\otimes b_0'[b_{k+1},\cdots , b_n]\\
 \end{split}
\end{equation*}
\end{proof}

\section{Suspended BV structure on the relative Hochschild homology  of commutative open Frobenius algebras}

In this section we exhibit a BV structure on the relative Hochschild homology of a  commutative symmetric open Frobenius algebra. In particular we introduce a product on the shifted relative Hochschild homology of symmetric commutative Frobenius algebras which should be an algebraic model for Goresky-Hingston \cite{GorHing} product on $H^*(LM,M)$. 
 
For a commutative DG-algebra $A$ the relative Hochschild chains are defined to be
 \begin{equation*}
 \widetilde{C}_*(A,A)= \oplus_{n\geq 1} A\otimes \bar{A}^{\otimes n}.
 \end{equation*}
equipped with the Hochschild differential. Since $A$ is commutative, $ \widetilde{C}_*(A,A)$ is stable under the Hochschild differential and fits in the split short exact sequence of complexes,
\begin{equation*}
\xymatrix{0 \ar[r] & (A, d_A) \ar[r] & C_*(A,A)\ar[r] &  \widetilde{C}_*(A,A)\ar[r] & 0}
\end{equation*}
The homology of  $\widetilde{C}_*(A,A)$ is denoted $\widetilde{HH}_*(A,A)$ and is called the relative Hochschild homology of $A$.

\begin{theorem}\label{thm-BV_hom}  The shifted relative Hochschild homology $\widetilde{HH}_*(A,A)[m-1]$ of a degree $m$ commutative symmetric  open Frobenius algebra $A$ is a BV algebra whose BV-operator is the Connes operator and the product at the chain level on $\tilde{C}_*(A,A)$  (before the shift ) is given by
\begin{equation*}
\begin{split}
x\ast y & = \sum_{(a_0b_0)} (-1)^{|(a_0b_0)'|+(|(a_0b_0)''|+|b_0|+1)(|a_1|+\cdots +|a_p| +p)} (a_0b_0)'[ a_1 , \dots , a_p ,  (a_0b_0)'' , b_1 , \dots , b_q]\\
& = \sum_{a_0)}(-1)^{|a'_0|+(|a''_0|+1)(|a_1|+\cdots +|a_p|+p )} a'_0[ a_1 , \dots , a_p ,  a_0''b_0 , b_1 , \dots , b_q]\\
& = \sum_{b_0} (-1)^{(m+1)|a_0|+ |b'_0|+(|b''_0|+|b_0|+1)(|a_1|+\cdots +|a_p| +p)}   a_0b_0'[ a_1 , \dots , a_p,  b_0'' , b_1 , \dots , b_q]\\
\end{split}
\end{equation*}
for $x=a_0 [a_1,\cdots ,a_p]$ and $y= b_0 [b_1, \cdots ,b_q] \in  \widetilde{C}_*(A) $.
\end{theorem}
\begin{proof}
Note that the identities above hold because $A$ is an open Frobenius algebra. Before the shift, the product is degree is a chain map and strictly associative of degree $m-1$, that is 
$$
D_{Hoch} (x\ast y)= (-1)^{m-1}( D_{Hoch} (x) \ast y+(-1)^{|x|}x\ast D_{Hoch}( y))
$$ 
and
$$
(x\ast y)\ast z=(-1)^{(m-1)x}(x\ast (y\ast z))
$$
It is noteworthy to mention that  commutativity is used in proving that $\ast$ is a chain map and associative, as it is shown below.  For instance, the term corresponding to first term of the external part of the Hochschild differential $D_{Hoch} (x\ast y)$ is
\begin{equation*}
\begin{split}
&\sum_{(a_0)}(-1)^{|a'_0|+(|a''_0|+1)(|a_1|+\cdots +|a_p| +p)+|a_0'|} a'_0 a_1 [a_2, \dots , a_p ,  a_0''b_0 , b_1 , \dots , b_q]=\\
&\sum_{(a_0)}(-1)^{(|1''|+1)(|a_1|+\cdots +|a_p| +p)+m|a_0|} a_0 1'a_1 [a_2, \dots , a_p ,  1''b_0 , b_1 , \dots , b_q]=\\
&\sum_{(a_0)}(-1)^{(|1''|+1)(|a_2|+\cdots +|a_p| +p-1)+m|a_0|+|a_1||1'|+(|a_1|+1)(|1''|+1)} a_0 a_11' [a_2, \dots , a_p ,  1''b_0 , b_1 , \dots , b_q]=\\
&\sum_{(a_0)}(-1)^{(|1''|+1)(|a_2|+\cdots +|a_p| +p-1)+m(|a_0|+|a_1|)+|a_1|+|1''|+1} a_0 a_11' [a_2, \dots , a_p ,  1''b_0 , b_1 , \dots , b_q]=\\
&\sum_{(a_0)}(-1)^{((a_0 a_1)''|+1)(|a_2|+\cdots +|a_p| +p-1)+|a_1|+|(a_0 a_1)''|+1} (a_0 a_1)' [a_2, \dots , a_p ,  (a_0 a_1)''b_0 , b_1 , \dots , b_q]=\\
&\sum_{(a_0)}(-1)^{((a_0 a_1)''|+1)(|a_2|+\cdots +|a_p| +p-1)+|a_1|+|a_0|+|a_1|+|(a_0 a_1)'|+m+1)} (a_0 a_1)' [a_2, \dots , a_p ,  (a_0 a_1)''b_0 , b_1 , \dots , b_q]=\\
&(-1)^{m-1}\sum_{(a_0)}(-1)^{|(a_0 a_1)'| +(|(a_0 a_1)''|+1)(|a_2|+\cdots +|a_p| +p-1)+|a_0|} (a_0 a_1)' [a_2, \dots , a_p ,  (a_0 a_1)''b_0 , b_1 , \dots , b_q],\\
\end{split}
\end{equation*}
 This is precisely the term which corresponds to first terms corresponding to the external part of the differential in $D_{Hoch}x \ast y$. The commutativity and cocommutativity  (or equivalently, commutativity and  being symmetric) of $A$ is required for the associativity of $\ast$ as well: We have
\begin{equation*}
\begin{split}
&(x\ast y) \ast z = (\sum(-1)^{|a'_0|+(|a''_0|+1)(|x|-|a_0|)} a'_0[ a_1 , \dots , a_p ,  a_0''b_0 , b_1 , \dots , b_q]) \ast  z=\\
&  \sum (-1)^{|a'_0|+(|a''_0|+1)(|x|-|a_0|)+(m+1)|a'_0|+|c_0'|+(|c_0''|+|c_0|+1)(|x|-|a_0|+|a_0''|+|y|+1))}\\
&\qquad \qquad \qquad \qquad \qquad \qquad\qquad\qquad a'_0c_0'[ a_1 , \cdots , a_p ,  a_0''b_0  , \cdots , b_q,c_0'',\cdots, c_r].
\end{split}
\end{equation*}
On the other hand
\begin{equation*}
\begin{split}
&(-1)^{(m-1)|x|}x\ast(y\ast z )=x\ast (\sum (-1)^{(m-1)|x|+ (m+1)|b_0|+|c_0'|+(|c_0''|+|c_0|+1)(|y|+|b_0|))}  \\
&\qquad \qquad \qquad \qquad \qquad \qquad\qquad\qquad b_0c_0'[ b_1  , \cdots , b_q,c_0'',\cdots, c_r]=\\
&  \sum (-1)^{(m-1)|x|+|a'_0|+(|a''_0|+1)(|x|-|a_0|)+(m+1)|b_0|+|c_0'|+(|c_0''|+|c_0|+1)(|y|+|b_0|))} \\
&\qquad \qquad \qquad \qquad \qquad \qquad\qquad\qquad (a'_0[ a_1 , \cdots , a_p ,  a_0''b_0 c_0' ,b_1 \cdots , b_q,c_0'',\cdots, c_r]=\\
&  \sum (-1)^{(m-1)|x|+|c_0'||b_0|+|a'_0|+(|a''_0|+1)(|x|-|a_0|)+(m+1)|b_0|+|c_0'|+(|c_0''|+|c_0|+1)(|y|+|b_0|))} \\
&\qquad \qquad \qquad \qquad \qquad \qquad\qquad\qquad (a'_0[ a_1 , \cdots , a_p ,  a_0''c_0'b_0 , \cdots , b_q,c_0'',\cdots, c_r]=\\
&  \sum (-1)^{(m-1)|x|+|c_0'||b_0|+|(a_0c_0)'|+(|(a_0c_0')''|+|c'_0|+1)(|x|-|a_0|)+(m+1)|b_0|+|c_0'|+(|c_0''|+|c_0|+1)(|y|+|b_0|))} \\
&\qquad \qquad \qquad \qquad \qquad \qquad\qquad\qquad (a_0c_0')'[ a_1 , \cdots , a_p ,  (a_0c_0')''b_0 , \cdots , b_q,c_0'',\cdots, c_r]=\\
&  \sum (-1)^{(m-1)|x|+|c_0'||a_0|+|c_0'||b_0|+|(c_0'a_0)'|+(|(c_0'a_0)''|+|c'_0|+1)(|x|-|a_0|)+(m+1)|b_0|+|c_0'|+(|c_0''|+|c_0|+1)(|y|+|b_0|))} \\
&\qquad \qquad \qquad \qquad \qquad \qquad\qquad\qquad (c_0'a_0)'[ a_1 , \cdots , a_p ,  (c_0'a_0)''b_0 , \cdots , b_q,c_0'',\cdots, c_r]=\\
&  \sum (-1)^{(m-1)|x|+|c_0'||a_0|+m|c_0'| +|c_0'||b_0|+|a_0'|+|c_0'|+(|a_0''|+|c_0'|+1)(|x|-|a_0|)+(m+1)|b_0|+|c_0'|+(|c_0''|+|c_0|+1)(|y|+|b_0|))} \\
&\qquad \qquad \qquad \qquad \qquad \qquad\qquad\qquad c_0'a_0'[ a_1 , \cdots , a_p ,  a_0''b_0 , \cdots , b_q,c_0'',\cdots, c_r]=\\
&  \sum (-1)^{(m-1)|x|+|c_0'||a'_0|+|c_0'||a_0|+m|c_0'| +|c_0'||b_0|+|a_0'|+|c_0'|+(|a_0''|+|c'_0|+1)(|x|-|a_0|)+(m+1)|b_0|+|c_0'|+(|c_0''|+|c_0|+1)(|y|+|b_0|))} \\
&\qquad \qquad \qquad \qquad \qquad \qquad\qquad\qquad a_0'c_0'[ a_1 , \cdots , a_p ,  a_0''b_0 , \cdots , b_q,c_0'',\cdots, c_r]=\\
\end{split}
\end{equation*}
After comparing the  coefficient the associativity follows.

However the commutative holds only up to homotopy, and the homotopy is given by
\begin{equation*}
\begin{split}
T(x,y):= \sum _{i= 0 }^q \sum_{(a_0)}(-1)^{\mu_i}  b_0[ b_1 , \cdots, b_{i} , a_{0}' , a_{1}, \cdots, a_{p}, a_{0}'',  b_{i+1}, \cdots, b_{q}],\\
\end{split}
\end{equation*}
where $\mu_i= (|b_0|+|b_1|\cdots |b_i|+i) (m+|a_0|+|a_1|+\cdots |a_p|+p)+  |a'_0|+(|a_0''|+1) (|a_1|+\cdots |a_p|+p)$.  By doing the computation we  see
that  for $i=0$, the first term of the external differential  of  $D_{Hoch} T(x,y)$ is precisely $x\ast y$, and  for $i=q$ the last term of the external differential $D_{Hoch} T(x,y)$  is 
$-(-1)^{|x||y|+m-1}y\ast x$. For the latter, one has to use the commutativity  and cocommutativity.

To prove that the 7-term relation holds, we adapt once again Chas-Sullivan's \cite{CS1} idea to a simplicial situation. First we define the Gerstenhaber bracket directly 
$$\{x,y\}:=T(x,y)+ (-1)^{m-1+|x||y|} T(y,x)$$ 
Next we prove that the bracket $\{-,-\}$ is homotopic  $B(x\ast y)-(-1)^{m-1} (B x \ast y +(-1)^{|x|} x \ast By)$. 
For that we decompose $B (x\ast y)$ in two pieces:
\begin{equation*}
\begin{split}
B_{1} (x,y):= \sum_{j=1}^{q}  \sum_{(a_0b_0)}\pm1[b_{j+1}, \dots  b_{q}, (a_0b_0)', a_1, \dots, a_p ,  (a_0b_0)'' , b_1 , \dots , b_j],
\end{split}
\end{equation*}
\begin{equation*}
\begin{split}
 B_{2} (x,y):=  \sum_{j=1}^{p}  \sum_{(a_0b_0)}\pm1[a_{j+1},\dots  a_{p}, (a_0b_0)'', b_1, \dots, b_q ,  (a_0b_0)' , a_1 , \dots , a_j],
\end{split}
\end{equation*}
so that  $B=B_1+B_2$.  The homotopy between $T(x ,y)$ and $ B_1(x,y)-(-1)^{m-1+|x|}x\ast B y$ is given by
\begin{equation*}
\begin{split}
H(x,y) &= \sum_{0\leq j\leq i\leq q}  \sum_{(a_0)} (-1)^{s_{i,j}}1 [b_{j+1}, \cdots , b_{i} ,  a_{0}' , a_{1}, \cdots, a_{p}, a_{0}'' , b_{i+1},\cdots , b_q,  b_{0}, \cdots, b_{j}].
\end{split}
\end{equation*}
where $s_{i,j}= (|b_0|+|b_1|\cdots |b_i|+i) (m+|a_0|+|a_1|+\cdots |a_p|+p)+  |a'_0|+(|a_0''|+1) (|a_1|+\cdots |a_p|+p) + (|b_0|+|b_1|+\cdots +|b_j|+i+1) (|b_{j+1}|\cdots |b_i|+m+|a_0|+|a_1|\cdot |a_p|+|b_{i+1}|\cdots |b_q|+ q-j+i+2)$

The terms  in  $D_{Hoch} H-(1)^{m}  H ( D_{Hoch}\otimes 1+ 1\otimes  D_{Hoch})$ corresponding to $j=0$, $j=i$ and $i=q$ are respectively, $T(x,y)$,  $(-1)^{m-1+|x|} x\ast B y$ and $-B_1(x,y)$. 
Similarly for $(-1)^{m-1+|x||y|}T(y,x)$ and $B_2(x\cdot y) -(-1)^{m-1}B x \ast y$.  Therefore we have proved that on $\widetilde{HH}_*(A,A)$  the bracket $\{-,-\}$ is the deviation of $B$ from being a derivation for $\ast$.

Now proving the 7-term relation is equivalent to proving the Leibniz rule for the bracket and the product $\ast$, \textit{i.e.}
$$
\{x,y\ast z\}= \{x,y\}\ast z+ (-1)^{(m+|x|)|y|}y\ast \{x,z\}.
$$
It turns out that at the chain level  $T(x, (y \ast z))= \ T(x, y)\ast z+(-1)^{(m+|x|)|y|}  y\ast T(x, z)$ and $T((y\ast z), x)$ is homotopic to
$(-1)^{|x|(|z|+m-1)}T(y, x) \ast z  + (-1)^{m(|x|-|y|)+|x|}y\ast T(z,x) $ using the homotopy

\begin{eqnarray*}
H_{3}(x,y,z) &=&\sum \pm a_0 [a_1, \dots , a_i, b_0', b_1 , \dots , b_p,b_0'', a_{i+1}\dots , a_j , c_0', c_1, \dots, c_r, c_0'', a_{j+1}, \dots ,a_p].
\end{eqnarray*}
Here $z=c_0[c_1, \dots , c_p]$.  This proves that the Leibniz rule holds up to homotopy.
\end{proof}
\appendix

\section{Ten commandments for signs}  
\begin{numlist}
\item \textbf{Morphism}: A  $\kk$-linear map $f$ of degree $|f|$ between differential $\kk$-modules $A$ and $B$ is said to be a morphism of differential graded $\kk$-modules if 
\begin{equation}\label{eq-diff}
fd_A=(-1)^{|f|}d_Bf 
\end{equation}
\item  \textbf{Tensor product of morphisms}: We use the following sign rule for  the tensor product of graded maps $f \in\Hom_{\kk}(A^{\otimes p} ,M)$ and   $g \in \Hom_{\kk}(A^{\otimes q} ,N)$

\begin{equation} \label{sign1}
(f\otimes g)(a_1\otimes \cdots \otimes a_{p+q})= (-1)^{|g|(|a_1|+\cdots +|a_p|)}f(a_1\otimes \cdots \otimes a_p)\otimes g(a_{p+1}\otimes\cdots \otimes a_{p+q})
\end{equation}
So as a result  the associativity  condition $\mu(\mu\otimes 1)=\mu(1\otimes \mu)$ for a binary operator $\mu:A^{\otimes 2}\to A$  of degree $m$ means that
\begin{equation} \label{eq-asso}
\mu (\mu(a \otimes b)\otimes c))=(-1)^{ma}\mu(a,\mu(b\otimes c))
\end{equation}

\item \textbf{Tensor product of algebra:} If $A$ and $B$ are differential $\kk$-modules  then $A\otimes B$ is also a differential $\kk$-module whose differential is give by $d_A\otimes 1+1\otimes d_B$.  If $A$ and $B$ are differential graded algebras then $A\otimes B$ is also a differential graded algebra whose product is defined by 
 $$(a\otimes x)(b\otimes y)=(-1)^{|b||x|}ab\otimes xy.$$

An important case is when $B=A^{op}$. The differential $A\otimes A^{op}$-modules are precisely differential $A$-bimodules.

\item \textbf{Differential of the dual}: The dual $\kk$-module $A^\vee=\hom_{\kk}(A,\kk)$ is negatively graded \emph{i.e.}
$A^\vee_{-i}= \Hom_{\kk}(A_i,\kk)$ and equipped with the differential $d^\vee$ which is defined by $d_{A^\vee}(\alpha)(x)=-(-1)^{|x|}\alpha(d_A(x))=(-1)^{|\alpha|}\alpha(d_A(x))$, $\alpha \in A^\vee$, and  also of degree one.  
Our choice of sign makes the evaluation map $ev: A\otimes A^\vee \to \kk$ a chain map of degree zero. 
We apply the same rule for a general $A$-bimodule $M$. That is  $M^\vee_{-i}=\hom_{\kk}(M_i,\kk)$ is equipped with the differential $d_{M^\vee}\phi=(-1)^{|\phi|}\phi \circ d_M$.

\item \textbf{Transpose of morphisms}: If $f:A\to B$ is a homogeneous morphism of differential graded $\kk$-modules then the induced map  $f^\vee$ on the $\kk$-duals is defined by 
\begin{equation*}
f^\vee(\phi)= (-1)^{|\phi||f|} f \circ \phi
\end{equation*}

Using the previous sign conventions $f^\vee$ is also a (homogeneous) morphism of differential graded $\kk$-module of degree $|f|$.

\item  \textbf{ Module structure of the dual}: There is a natural $A$ right and left $A$-module structure on $A^\vee$ give by $x.\alpha : y\to \alpha(yx)$ and $\alpha.x:y\to (-1)^{|x|}\alpha (xy)$. The maps $(x,\alpha)\to x.\alpha$ and  $(x,\alpha)\to \alpha .x$  are chain maps.  

We apply the same rule for a general $A$-bimodule $M$. That is  $M^\vee$ is equipped with the $A$-bimodule structure $(x.\alpha)(y):=\alpha(yx)$ and  $(\alpha.x)(y):=(-1)^{|x|}\alpha (xy)$, where $\alpha \in M^\vee$.

\item  \textbf{Dual of tensor product}:  For each pair of right $A$-module $M$ and left $A$-module $ N$, note that $M\otimes N$ is an $A$-bimodule, so is $(M\otimes N)^\vee$.
There is a natural inclusion of $A$-bimodules the inclusion  $i_{N,M:}N^\vee\otimes M^\vee \hookrightarrow (M\otimes N)^\vee  $ given by $\phi_1\otimes \phi_2\to (-1)^{|\phi_1||\phi_2|}\phi_2\otimes \phi_1$. 

\item  \textbf{Transpose and tensor product}: If $f: A^{\otimes l}\to A^{\otimes k}$ is graded  (homogenous) morphism  of degree  $|f|$ then then adjoint map $f^\vee: (A^\vee )^{\otimes k} \to (A^\vee )^{\otimes l} $ is  defines  by
\begin{equation} \label{sign2}
f^\vee (\phi_1\otimes \cdots \otimes \phi_k)=(-1)^{(|\phi_1|+\cdots +|\phi_k|) |f|}  (\phi_1\otimes \cdots \otimes \phi_k)\circ  f.
\end{equation}

\item \textbf{Shift of Grading }: The shift of the degree by $m$  to the right is denoted by $s_{m}:A\to A$ and $\deg (s_{m}(a))= \deg (a) + m$, or in other words $A[m]_k=A_{k-m}$. Using the shift operation one can  pullback other operations, for instance a product $\mu: A\otimes A\to A$  is pulled back to 
\begin{equation} \label{sign3}
\mu_{m} :=s_{m}^*(\mu)= s_{m} \circ \mu \circ (s_{m}^{-1} \otimes s_m^{-1}),
\end{equation}
or more explicitly $\mu_m(s_{m}(a),s_{m}(b))=(-1)^{m|a|} s_{m}\mu(a,b)$.

So if $\mu$ is of degree  $m$ then $\mu_m$ is of degree zero and  $\mu_m$ is associative (of degree zero)  product on  $A[m]$ if $\mu$ is associative  (of degree $m$) 
\emph{i.e.}
$$
\mu(\mu\otimes 1)=\mu(1\otimes  \mu)
$$

$$
\mu(\mu(a,b)c)=(-1)^{m|a|+m}\mu(a, \mu(b,c))
$$
The commutativity condition for $\mu'$ is equivalent to  $\mu=(-1)^{m}\mu\circ \tau$ where $t:A\otimes A\to A\otimes A$ is given by $\tau(x\otimes y)=(-1)^{|x||y|}y\otimes x$.

Similarly the coassociativity rule for  coprodcut $\delta$ of degree $m$ is obtained by writing down the usual associative rules of  the the degree zero coproduct $\delta'= (s_{-m}\otimes s_{-m} )\delta s^{-1}_{-m}$, which translates to 
$$
(\delta\otimes 1)\delta=(-1)^m(1\otimes \delta)\delta.
$$

Using the same argument the equations defining a degree $m$  (right and left) counit  for $\delta$ are  $\id= \sim \circ (\eta \otimes 1)\delta'= \sim \circ (1 \otimes \eta)\delta'$ where $\sim$ stands for the natural isomorphisms
$ A\otimes _\kk\kk \simeq A$ and $ \kk \otimes _\kk A \simeq A$. Said explicitly  $x=\sum_{(x)}  (-1)^{m|x'|}\eta(x')x''=\sum_{(x)}  (-1)^{m|x'|}\eta(x'')x'$.  

The cocommutativity condition  for $\delta'$ becomes  $\delta=(-1)^m\tau\delta$.

\item \textbf{Grading shift and derivations} A degree $|D|$ derivation $D:A\to A$ for a degree $m$ bilinear map $\mu: A\otimes A\to$ is $\kk$-linear map which satisfies the identity
$$
D\mu=(-1)^{|D||\mu|}\mu (D\otimes 1+1\otimes D) 
$$
After a shift  of degree to the right, $ D$ is a still derivation of degree $|D|$  on $A[m]$ with respect to the degree zero binary operation $\mu_m$. In particular if  $|D|=1$, $D^2=0$ and $\mu$ is associative then  $(A[m], \mu_m,D)$ is a differential graded associative algebra.
\end{numlist}
\bibliography{Hoch-Op-Frob}{}

\providecommand{\bysame}{\leavevmode\hbox to3em{\hrulefill}\thinspace}
\providecommand{\MR}{\relax\ifhmode\unskip\space\fi MR }
\providecommand{\MRhref}[2]{%
  \href{http://www.ams.org/mathscinet-getitem?mr=#1}{#2}
}
\providecommand{\href}[2]{#2}
\begin{thebibliography}{GH09b}

\bibitem[CG04]{CG}
Ralph~L. Cohen and V{\'e}ronique Godin, \emph{A polarized view of string
  topology}, Topology, geometry and quantum field theory, London Math. Soc.
  Lecture Note Ser., vol. 308, Cambridge Univ. Press, Cambridge, 2004,
  pp.~127--154.

\bibitem[CJ02]{CJ}
Ralph~L. Cohen and John D.~S. Jones, \emph{A homotopy theoretic realization of
  string topology}, Math. Ann. \textbf{324} (2002), no.~4, 773--798.

\bibitem[Con85]{connes}
Alain Connes, \emph{Noncommutative differential geometry}, Inst. Hautes
  \'Etudes Sci. Publ. Math. (1985), no.~62, 257--360.

\bibitem[Cos07]{CostelloCY}
Kevin Costello, \emph{Topological conformal field theories and {C}alabi-{Y}au
  categories}, Adv. Math. \textbf{210} (2007), no.~1, 165--214.

\bibitem[CS]{CS1}
Moira Chas and Dennis Sullivan, \emph{String topology}, arxive math/9911159.

\bibitem[CS04]{CS2}
\bysame, \emph{Closed string operators in topology leading to {L}ie bialgebras
  and higher string algebra}, The legacy of {N}iels {H}enrik {A}bel, Springer,
  Berlin, 2004, pp.~771--784.

\bibitem[F{\'e}l]{FTBV}
Yves F{\'e}lix, \emph{Rational bv-algebra in string topology}, Bull. Soc. Math.
  France \textbf{136}, no.~2, 311--327.

\bibitem[FT08]{FT2}
Yves F{\'e}lix and Jean-Claude Thomas, \emph{Rational {BV}-algebra in string
  topology}, Bull. Soc. Math. France \textbf{136} (2008), no.~2, 311--327.

\bibitem[Ger63]{Gers}
Murray Gerstenhaber, \emph{The cohomology structure of an associative ring},
  Ann. of Math. (2) \textbf{78} (1963), 267--288. \MR{0161898 (28 \#5102)}

\bibitem[Get94]{Getz}
E.~Getzler, \emph{Two-dimensional topological gravity and equivariant
  cohomology}, Comm. Math. Phys. \textbf{163} (1994), no.~3, 473--489.

\bibitem[GH09a]{GH}
Mark Goresky and Nancy Hingston, \emph{Loop products and closed geodesics},
  Duke Math. J. \textbf{150} (2009), no.~1, 117--209.

\bibitem[GH09b]{GorHing}
\bysame, \emph{Loop products and closed geodesics}, Duke Math. J. \textbf{150}
  (2009), no.~1, 117--209. \MR{2560110 (2010k:58021)}

\bibitem[Hat02]{Hatcherbook}
Allen Hatcher, \emph{Algebraic topology}, Cambridge University Press,
  Cambridge, 2002.

\bibitem[Jon87]{Jones}
John D.~S. Jones, \emph{Cyclic homology and equivariant homology}, Invent.
  Math. \textbf{87} (1987), no.~2, 403--423.

\bibitem[Kau07]{Kauf1}
Ralph~M. Kaufmann, \emph{Moduli space actions on the {H}ochschild co-chains of
  a {F}robenius algebra. {I}. {C}ell operads}, J. Noncommut. Geom. \textbf{1}
  (2007), no.~3, 333--384.

\bibitem[Kau08]{Kauf2}
\bysame, \emph{Moduli space actions on the {H}ochschild co-chains of a
  {F}robenius algebra. {II}. {C}orrelators}, J. Noncommut. Geom. \textbf{2}
  (2008), no.~3, 283--332.

\bibitem[Kos85]{KoszulBV}
Jean-Louis Koszul, \emph{Crochet de {S}chouten-{N}ijenhuis et cohomologie},
  Ast\'erisque (1985), no.~Numero Hors Serie, 257--271, The mathematical
  heritage of {\'E}lie Cartan (Lyon, 1984).

\bibitem[KS09]{KS}
M.~Kontsevich and Y.~Soibelman, \emph{Notes on {$A_\infty$}-algebras,
  {$A_\infty$}-categories and non-commutative geometry}, Homological mirror
  symmetry, Lecture Notes in Phys., vol. 757, Springer, Berlin, 2009,
  pp.~153--219.

\bibitem[LS08]{LamStan}
Pascal Lambrechts and Don Stanley, \emph{Poincar\'e duality and commutative
  differential graded algebras}, Ann. Sci. \'Ec. Norm. Sup\'er. (4) \textbf{41}
  (2008), no.~4, 495--509.

\bibitem[Men09]{Me}
Luc Menichi, \emph{Batalin-{V}ilkovisky algebra structures on {H}ochschild
  cohomology}, Bull. Soc. Math. France \textbf{137} (2009), no.~2, 277--295.

\bibitem[Mer04]{Mer}
S.~A. Merkulov, \emph{De {R}ham model for string topology}, Int. Math. Res.
  Not. (2004), no.~55, 2955--2981.

\bibitem[Tra08]{Tradler}
Thomas Tradler, \emph{The {B}atalin-{V}ilkovisky algebra on {H}ochschild
  cohomology induced by infinity inner products}, Ann. Inst. Fourier (Grenoble)
  \textbf{58} (2008), no.~7, 2351--2379.

\bibitem[TZ06]{TZ}
Thomas Tradler and Mahmoud Zeinalian, \emph{On the cyclic {D}eligne
  conjecture}, J. Pure Appl. Algebra \textbf{204} (2006), no.~2, 280--299.

\bibitem[WW]{WahWest}
Nathalie Wahl and Craig Westerland, \emph{Hochschild homology of structured
  algebras}, arxiv:1110.0651.

\end{thebibliography}

\bibliographystyle{amsalpha}

\end{document}